\documentclass[12pt,reqno]{amsart}
\usepackage{amsmath,amsfonts,graphicx,amscd,amssymb,epsf,amsthm,enumerate, mathrsfs,color,ltablex,blindtext}  
\usepackage[a4paper, margin=1in]{geometry}
\usepackage{times}
\usepackage[margin=10pt,font=small,labelfont=bf]{caption}
\usepackage{matlab-prettifier}
\usepackage{tcolorbox}
\theoremstyle{definition}
\newtheorem{theorem}{Theorem}[section]
\newtheorem{proposition}[theorem]{Proposition}
\newtheorem{lemma}[theorem]{Lemma}

\newtheorem{remark}[theorem]{Remark}
\numberwithin{equation}{section}
\numberwithin{equation}{section}

\usepackage{hyperref}
\hypersetup{
	colorlinks=true,
	linkcolor=blue,
	citecolor=blue,
}

\newcommand{\di}{\displaystyle}

\newcommand{\mb}{\mathbb}

\begin{document}
	
	\title[Explicit  Bounds of $|\zeta\left(1+it\right)|$]{Explicit  Bounds of $\pmb{|\zeta\left(1+it\right)|}$}
	\author{Eunice Hoo Qingyi and Lee-Peng Teo}
	\address{Department of Mathematics, Xiamen University Malaysia\\Jalan Sunsuria, Bandar Sunsuria, 43900, Sepang, Selangor, Malaysia.}
	\email{MAM2209001@xmu.edu.my, lpteo@xmu.edu.my}
	
	\begin{abstract}
		In this work, we study a bound of the form $|\zeta(1+it)|\leq v\log t$ for  $t\geq t_0$. We show that the exponential sum method with second order derivatives can achieve  any $v >\frac{1}{2}$ as long as $t_0$ is sufficiently large.    Using the Riemann--Siegel formula and numerical computations, we show that when $t\geq e$, 
		\[|\zeta(1+it)|\leq\frac{1}{2}\log t+0.6633. \] This allows us to show that   
		\[|\zeta(1+it)|\leq 0.6443 \log t\quad \text{when}\;\;t\geq e. \] This is the best possible result of the form $|\zeta(1+it)|\leq v\log t$ that holds for all  $t\geq e$, as the equality is achieved when $t=17.7477$.  
		
	\end{abstract}
	\subjclass{11M06, 11L07}
	\keywords{Riemann zeta function, exponential sums, Riemann--Siegel formula, explicit bounds}
	\maketitle
	\section{Introduction}
	
	The Riemann zeta function $\zeta(s)$ is defined as
	\[\zeta(s)=\sum_{n=1}^{\infty}\frac{1}{n^s}\] when $s=\sigma+it$ is a complex number with $\text{Re}\, s>1$. It has an analytic continuation to the complex plane $\mb{C}$ except for a simple pole at $s=1$.  For more details on the Riemann zeta function, one can refer to the books   \cite{Edwards, Titchmarsh, Ivic_Book}.
	
	 It is a classical result that $\zeta(1+it)\neq 0$ for any real number $t\neq 0$. In this work, we consider   upper bounds for $|\zeta(1+it)|$. This has been a problem of interest for more than 100 years. 
		It was first shown by Mellin \cite{Mellin_1902} that
	\[\zeta(1+it)=O(\log t).\]
	In  \cite{Weyl_1921}, Weyl proved that
	\[\zeta(1+it)=O\left(\frac{\log t}{\log\log t}\right).\] 
	Using an improved form of his own mean value theorem, Vinogradov \cite{Vinogradov_58}  showed that
	\[\zeta(1+it)=O\left((\log t)^{\frac{2}{3}}\right).\] If we believe that the Riemann hypothesis is true, the order is much smaller.  Littlewood \cite{Littlewood_1925} proved that \[\zeta(1+it)=O(\log\log t)\] if the Riemann hypothesis is true (see  also Section 14 of \cite{Titchmarsh}).
	For  unconditional explicit upper bounds, Landau \cite{Landau_03} showed that  \[|\zeta(1+it)|\leq 2\log t \hspace{1cm}\text{when}\;t\geq 10.\] Later Backlund \cite{Backlund} proved that  
	\[|\zeta(1+it)|\leq\log t\hspace{1cm}\text{when}\;t\geq 50.\]
	In   \cite{Ford_2002}, Ford refined Vinogradov's method and  showed  that
	\[|\zeta(1+it)|\leq 76.2 \left(\log t\right)^{\frac{2}{3}}\hspace{1cm}\text{when}\;t\geq 3.\]In fact, he obtained a bound of $|\zeta(\sigma+it)|$ for all $\frac{1}{2}\leq\sigma\leq 1$ and $t\geq 3$. 
	Trudgian \cite{Trudgian_2014} noticed that the constant $76.2$ can be reduced to $62.6$ when $\sigma=1$. In the same paper  \cite{Trudgian_2014}, Trudgian claimed  that 
	\begin{equation}\label{20241022_1}|\zeta(1+it)|\leq \frac{3}{4}\log t\hspace{1cm}\text{when}\;t\geq 3.\end{equation}
	Unfortunately, the result was obtained using a wrong lemma in \cite{ChengGraham2004}.
	In   \cite{Patel2022}, Patel pointed out the error and showed that
	\[|\zeta(1+it)|\leq\min\left\{\log t, \frac{1}{2}\log t+1.93, \frac{1}{5}\log t+44.02\right\}.\]
	Notice that 
	\[\log t\geq \frac{1}{2}\log t+1.93\] if and only if $t\geq 47.47$, and 
	\[\frac{1}{2}\log t+1.93\geq \frac{1}{5}\log t+44.02\]if and only if $t\geq 8.54\times 10^{60}$.  Thus, the bound
	\[|\zeta(1+it)|\leq \frac{1}{5}\log t+44.02\] might not be that useful from the point of view of applications. 
	
	In this work, we  first consider the best upper bound we can obtain for $|\zeta(1+it)|$ using only the  representation of the Riemann zeta function (see Section \ref{sec2})
\begin{align*}\zeta(s) &=\sum_{n\leq x}\frac{1}{n^{s}}+\frac{x^{1-s}}{s-1} +B_1 (\{x\}   )x^{-s}+\frac{s}{2}B_2(\{x\})x^{-s-1}-\frac{s(s+1)}{2}\int_x^{\infty}\frac{B_2(\{u\})}{u^{s+2}}du\end{align*} and the  exponential sum methods. We obtain   bounds of the form
\[|\zeta(1+it)|\leq v\log t\hspace{1cm}\text{when}\;t\geq t_0\]for various $v$ and $t_0$, with $v>\frac{1}{2}$ and $t_0$ sufficiently large.  When $t_0=10^6$, the best we can obtain  is
\[|\zeta(1+it)|\leq 0.7421\log t\hspace{1cm}\text{when}\;t\geq 10^6.\]
Then numerical computations  show  that the same still holds when $e\leq t\leq 10^6$.  In fact, numerical computations show that when $e\leq t\leq 10^6$, $|\zeta(1+it)|/\log t$ achieves the maximum value $0.6443$ when $t=17.7477$. In principle, we can show that  $|\zeta(1+it)|\leq 0.6443 \log t$ holds for all $t\geq e$ if we can perform numerical computations of $\zeta(1+it)$ for $t\leq 10^{10}$. Due to the limitation of our computer, we prefer to use a different approach.
Using the representation of the Riemann zeta function provided by the Riemann--Siegel formula, we can show that
	\[ |\zeta(1+it)|\leq\frac{1}{2}\log t+0.6633 \hspace{1cm}\text{when}\;t\geq 10^6.\]
	Numerical computations   show that this inequality still holds when $e\leq t\leq 10^6$. In other words, we find that 
	\[ |\zeta(1+it)|\leq\frac{1}{2}\log t+0.6633 \hspace{1cm}\text{when}\;t\geq e.\]
Using this result, together with numerical computations, we obtain 
\begin{align*}
|\zeta(1+it)|&\leq 0.6443 \log t\hspace{1cm}\text{when}\;t\geq e,\\
|\zeta(1+it)|&\leq 0.5480\log t\hspace{1cm}\text{when}\;t\geq  652.3704.
\end{align*}
The following theorem summarizes the main results of this work.
\begin{theorem}\label{maintheorem}
		For $t\geq e$, we have
		\begin{equation}\label{20250612_8} |\zeta(1+it)|\leq\frac{1}{2}\log t+0.6633\end{equation} and
		\begin{equation}\label{20250612_7}|\zeta(1+it)|\leq 0.6443 \log t.\end{equation}
		 The result \eqref{20250612_7} is the best possible for a bound of the form $|\zeta(1+it)|\leq v\log t$ that holds for all $t\geq e$. When $t\geq 100$, the bound \eqref{20250612_8} is better than the bound \eqref{20250612_7}. When $t\geq 652.3704$, we can obtain a better bound of the form  $|\zeta(1+it)|\leq v\log t$, which says that
		\begin{equation}\label{20250613_7}|\zeta(1+it)|\leq 0.5480 \log t\hspace{1cm}\text{when}\;t\geq 652.3704.\end{equation} When $ 652.3704 \leq t\leq 10^6$, the bound  \eqref{20250613_7} is better than the bound  \eqref{20250612_8}. 
	\end{theorem}
	In \cite{Yang_3}, it was shown that
	\[|\zeta(1+it)|\leq 1.7310\frac{\log t}{\log\log t}\hspace{1cm}\text{for all}\;t\geq 3.\]
	When $t\leq 6.05\times 10^{11}$, 
	\[\frac{1}{2}\log t+0.6633\leq  1.7310\frac{\log t}{\log\log t}.\] 
	Thus, for $t\leq 6.05\times 10^{11}$, our results are still  better than that obtained in \cite{Yang_3}.
	
	\vspace{0.5cm}
	\noindent
	{\bf Acknowledgment.} This work is supported by the Xiamen University 
	Malaysia Research Fund  XMUMRF/2021-C8/IMAT/0017.  Part of this work is contained in the M.Sc thesis of the first author. We are grateful to Professor T. Trudgian and the reviewer for their valuable comments.

	
	\section{Classical Representation of the Riemann Zeta Function}\label{sec2}
	In this section, we refine a classical formula  that can be used to compute $\zeta(1+it)$.	
	Given a real number $x$, we let $\lfloor x\rfloor$ denote the floor of $x$, and let
	$\{x\}=x-\lfloor x\rfloor$. Then we have
	\[x-1<\lfloor x\rfloor\leq x\quad  \text{and}\quad 0\leq \{x\}<1.\]
	
	The first and second Bernoulli polynomials are given respectively by
	\[B_1(x)=x-\frac{1}{2},\hspace{1cm} B_2(x)=x^2-x+\frac{1}{6}.\]
	Using the Euler-Maclaurin summation formula (see for example, \cite{Ivic_Book, Montgomery_Book}), one has the following integral representation of the Riemann zeta function $\zeta(s)$ for $\text{Re}\,s>-1$:
	\begin{equation} \label{20241022_3}\begin{split}
			\zeta(s)&=\sum_{n\leq x}\frac{1}{n^{s}}+\frac{x^{1-s}}{s-1} +B_1 (\{x\}   )x^{-s}+\frac{s}{2}B_2(\{x\})x^{-s-1}\\&\quad -\frac{s(s+1)}{2}\int_x^{\infty}\frac{B_2(\{u\})}{u^{s+2}}du.\end{split}
	\end{equation}
	By choosing $x$ large enough, the first line can be used to  give an approximation of $\zeta(1+it)$ with the error given by the integral in the second line. To obtain a smaller error bound, we   replace the polynomial $B_2(x)=x^2-x+\frac{1}{6}$ with the polynomial $h(x)=x^2-x+\frac{1}{8}$. Since $B_2(x)-h(x)=\frac{1}{24}$ is a constant, and
\[\frac{s(s+1)}{2}\int_x^{\infty}\frac{1}{u^{s+2}}du=\frac{s}{2} x^{-s-1},\] we obtain   the following integral representation.
	\begin{theorem} \label{20241022_4}
		Let $s=\sigma+it$ and let $h(x)=  x^2-x+\frac{1}{8}$. If $s\neq 1$, $\sigma>-1$, then for any $x\geq 1$, we have
		\begin{equation} \label{20241022_2}\begin{split}
				\zeta(s)&=\sum_{n\leq x}\frac{1}{n^{s}}+\frac{x^{1-s}}{s-1} +\left(\{x\}-\frac{1}{2} \right)x^{-s}+\frac{s}{2}h(\{x\})x^{-s-1}\\&\quad -\frac{s(s+1)}{2}\int_x^{\infty}\frac{h(\{u\})}{u^{s+2}}du.\end{split}
		\end{equation} 
	\end{theorem}The choice of $h(x)$ over $B_2(x)$ is due to the fact that  
	\[ \max_{0\leq x\leq 1}| B_2(x)|= \frac{1}{6},\]while\[ \max_{0\leq x\leq 1}| h(x)|=\frac{1}{8}.\]
	 
	Using
	formula \eqref{20241022_2}, we obtain the following.
	\begin{theorem}\label{20241022_6}
		For $t>0$ and $N$  a positive integer, let
		\begin{equation}\label{20250529_1}g_N(t)=\sum_{n=1}^{ N}\frac{1}{n^{1+it}}+\frac{N^{-it}}{it} -\frac{1}{2} N^{-1-it}+\frac{1+it}{16}N^{-2-it}.\end{equation}
		Then
		\begin{equation}\label{20250613_2}
		\left|\zeta(1+it)-g_N(t)\right|\leq \frac{(1+t)(2+t)}{32N^2}.
		\end{equation}
	\end{theorem}
	\begin{proof}
		From  \eqref{20241022_2}, we find that when $t>0$ and $N$ is a positive integer, we have
		\begin{align*}
			\left|\zeta(1+it)-g_N(t)\right|&\leq \left|\frac{(1+it)(2+it)}{2}\right|\int_N^{\infty} \frac{|h(\{u\})|}{|u^{3+it}|}du\\
			&\leq 
		\frac{(1+t)(2+t)}{16}\int_N^{\infty}\frac{du}{u^3}\\&=\frac{(1+t)(2+t)}{32N^2}. \qedhere
		\end{align*}
		 
	\end{proof}
	 
	By Theorem \ref{20241022_6},
		for $t>0$, when $N$ is a large enough positive integer, the function
		$g_N(t)$ \eqref{20250529_1}
		can be used to approximate $\zeta(1+it)$.  The error is at most
		\[
		  \frac{(1+t)(2+t)}{32N^2}.
		\]
		To compute $\zeta(1+it)$ for $t$ in the interval $[t_0, T]$, choose a fixed spacing $h$ and generate the points $t_k=t_0+kh$ for $0\leq k\leq K$, where $K=\lfloor(T-t_0)/h\rfloor$.  The spacing $h$  should be small enough so that the  values $|\zeta(1+it_k)|$ computed can sufficiently reflect the behaviour of $|\zeta(1+it)|$ for $t$ in the interval $[t_0, T]$. Usually it is sufficient to take $h$ to be $0.01$. To control the accuracy of approximation, we fixed an error threshold $r$ and let $N$ be the smallest positive integer so that
		\[ \frac{(1+T)(2+T)}{32N^2}\leq r.\]Then   $g_N(t_k)$, $0\leq k\leq K$ are computed  and they gave values of  $|\zeta(1+it_k)|$ with error  at most $r$. The value of $r$ we take depends on the accuracy of  $|\zeta(1+it_k)|$ that we want to achieve. For plotting graphs over a large range of $t$, we take $r=0.005$. 

	The following is the MATLAB code.
\begin{tcolorbox}[colback=white, colframe=black] 
\begin{lstlisting}[style=Matlab-editor]
function zeta(t0,T,h,r)
	
t=t0:h:T; 	
N=ceil(sqrt((1+T)*(2+T)/(32*r)));
g=zeros(size(t));
	
for n=1:N
     g=g+1./n.^(1+1i*t);
end
	
g=g+N.^(-1i*t)./(1i*t)-1/2*N.^(-1-1i*t)+(1+1i*t)/16.*N.^(-2-1i*t); % g_N(t) 
	 
\end{lstlisting}
\end{tcolorbox}

	This algorithm has an advantage for being simple. However, the computations become very intensive when $T$ is large. In this work, we limit our computations to $T\leq 10^6$ only. For $e\leq t\leq 10^6$, we find that
	    $|\zeta(1+it)|/\log t$ achieves its maximum value 0.6443 when $t=17.7477$.

	In the coming section, we want to obtain explicit upper bounds of $|\zeta(1+it)| $ using the methods of exponential sums.  
Applying triangle inequality to \eqref{20241022_2}, together with the simple estimate
		\begin{align*}
			\left|\frac{(1+it)(2+it)}{2}\int_x^{\infty}\frac{ h(\{u\}) }{u^{3+it}}du\right|\leq \frac{(t+1)(t+2)}{32 x^2},
		\end{align*}we obtain immediately the following.
	\begin{theorem}
		For $x\geq 1$ and $t>0$,
		\begin{equation}\label{240114_1}|\zeta(1+it)| \leq \left|\sum_{n\leq x}\frac{1}{n^{1+it}}\right|+\frac{1}{t}+\frac{1}{2x}+\frac{t^2+5t+4}{32x^2}.\end{equation}
	\end{theorem}
	 
	To find an upper bound of $|\zeta(1+it)|$, it remains to bound the sum in \eqref{240114_1}.
	 A crude estimate  is given by
	\begin{equation}\label{20250529_2}\left|\sum_{n\leq x}\frac{1}{n^{1+it}}\right|\leq\sum_{n\leq x}\frac{1}{n}.\end{equation} The sum $\di  \sum_{n\leq x}\frac{1}{n}$ is classical.
	In \cite{Francis}, it has been shown that for $x\geq 1$, 
		\[\sum_{n\leq x}\frac{1}{n } \leq \log x+\gamma+\frac{1}{2x}-\frac{1}{12x^2}+\frac{1}{64x^4},\]where $\gamma$ is the Euler-Mascheroni constant.
	  However, we can derive a simpler formula as follows.
	\begin{lemma}\label{240112_4}
		For $x\geq 1$,
		\begin{equation}\label{240112_3}
			\sum_{n\leq x}\frac{1}{n } \leq \log x+\gamma+\frac{1}{x}.
		\end{equation}
		
	\end{lemma}
	 
	\begin{proof}
		Using Euler's summation formula, one can obtain the harmonic sum bound
		\[ \sum_{n\leq N}\frac{1}{n}\leq\log N+\gamma+\frac{1}{N}\] when $N$ is an integer (see for example, \cite{Apostol_1}).
		Now for any real number $x\geq 1$, let $N=\lfloor x\rfloor$. It is easy to check that the function $\log x+\frac{1}{x}$ is increasing if $x\geq 1$. By definition, $N\leq x$. Hence,
		\[ \sum_{n\leq x}\frac{1}{n }=\sum_{n\leq N}\frac{1}{n} \leq\log N+\gamma+\frac{1}{N}\leq \log x+\gamma+\frac{1}{x}. \qedhere\]
		 
	\end{proof}
 
	\bigskip
	\section{Bounds of $ |\zeta(1+it)|$ Using Exponential Sums}\label{secexponential}
	
	To obtain an upper bound of $|\zeta(1+it)|$, one of the elementary methods is to use \eqref{240114_1}. One then needs to find a bound of 
	\[ \sum_{n\leq x}\frac{1}{n^{1+it}}.\]

	Let $a$ and $b$ be   real numbers with $a<b$. If $f(n)$ is real for all integers $n$ in the interval $I=(a, b]$, a sum of the form
	\[\sum_{a<n\leq b}e^{2\pi i f(n)} \]  is called an exponential sum.   A good reference for exponential sums is the book \cite{Kolesnik_book}.
	A classical result of Kuzmin and Landau \cite{Kuzmin, Landau_28} says that
	if $f$ is twice continuously differentiable and 
	$\lambda\leq |f''(x)|\leq\alpha\lambda$ for some positive contants $\lambda$ and $\alpha$, then 
	\begin{equation}\label{20241024_2}\sum_{a<n\leq b}e^{2\pi i f(n)}=O\left(\alpha|I|\lambda^{\frac{1}{2}}+\lambda^{-\frac{1}{2}}\right),\end{equation}where $|I|=b-a$.
	 To obtain explicit bounds for $|\zeta(1+it)|$, one needs to compute the implied constants in \eqref{20241024_2}. 
	This has been done by several authors \cite{Platt_Trudgian, Patel2022, Yang_2, Yang_1} in the case where $a$ and $b$ are integers. For our applications, we consider the general case where $a$ and $b$ are any real numbers with $a<b$.  In \cite{Patel2022}, Patel made a remark that one can use an observation in \cite{Platt_Trudgian} to improve his result to the following. If $N$ and $L$ are integers, $f:[N+1,N+L]\to\mb{R}$ is a twice   continuously differentiable function, $V$ and $W$ are positive numbers with $V<W$, and
	\[\frac{1}{W}\leq |f''(x)|\leq\frac{1}{V}\hspace{1cm}\text{for all}\;x\in [N+1, N+L],\]
	then 
	\begin{equation}\label{20250914_1}\left|\sum_{n=N+1}^{N+L}e^{2\pi i f(n)}\right|\leq \frac{4(L-1)}{\sqrt{\pi}}\frac{\sqrt{W}}{V}+\frac{8\sqrt{W}}{\sqrt{\pi}}+\frac{ L-1 }{V}+3.\end{equation}
	
	In the following, we consider arbitrary real numbers $a$ and $b$ with $a<b$. If $\lfloor b \rfloor=\lfloor a\rfloor=0$, there is no integer $n$ satisfying $a<n\leq b$ and so the sum is vacuous. If $\lfloor b \rfloor- \lfloor a\rfloor\geq 1$,  take  
	\[N=\lfloor a \rfloor,\quad L=\lfloor b\rfloor-\lfloor a\rfloor,\quad W=\frac{1}{\lambda},\quad V=\frac{1}{\alpha \lambda}\]in \eqref{20250914_1}. Since
	\[L\leq b-a+1,\]  we obtain the following.
	\begin{theorem}\label{240113_1}
		Let $a$ and $b$ be real numbers such that $a<b$, and let $I=(a,b]$. If $f:I\to\mb{R}$ is a twice   continuously differentiable function, and  there exist $\lambda>0$ and $\alpha\geq 1$ such that
		\begin{equation*}\lambda\leq |f''(x)|\leq \alpha\lambda\hspace{1cm}\text{for all}\;x\in I,\end{equation*}  
		then
		\begin{equation}\label{20250529_3}\left|\sum_{n\in I}e^{2\pi i f(n)}\right|\leq   \frac{4}{\sqrt{\pi}}\alpha|I|\lambda^{\frac{1}{2}}+\frac{8}{\sqrt{\pi}}\lambda^{-\frac{1}{2}}+\alpha|I|\lambda+3,\end{equation}where $|I|=b-a$.

	\end{theorem}

Applying  Theorem \ref{240113_1}, we obtain the following.
	
	\begin{proposition}\label{240113_2}
		For any $t>0$,  and any real numbers $a$ and $b$ such that $0< a<b$,   we have
		\[\left|\sum_{a<n\leq b}\frac{1}{n^{ 1+it}}\right|\leq \frac{4\sqrt{2}(b-a)}{\pi a^2}\sqrt{t} -\frac{2\sqrt{2}}{\pi a }\sqrt{t} \log\frac{b}{a}+\frac{8\sqrt{2}}{\sqrt{t}}\log\frac{b}{a}+\frac{8\sqrt{2}}{\sqrt{t}}+\frac{t}{2\pi a^2}\log\frac{b}{a}+\frac{3}{a}.\]
	\end{proposition}
	\begin{proof}
		
		For $u\in [a, b]$, let
		\[S(u)=\sum_{a<n\leq u}\frac{1}{n^{  it}}.\]
		Then $S(a)=0$. For fixed $t>0$, we take
		\[f(x)=-\frac{t}{2\pi}\log x,\quad x>0.\]
		Then
		\[\frac{t}{2\pi u^2}\leq f''(x)\leq \frac{t}{2\pi a^2}\hspace{1cm}\text{for}\quad a\leq x\leq u.\]
		Applying Theorem \ref{240113_1} with 
		\[\lambda=\frac{t}{2\pi u^2}, \hspace{1cm} \alpha=\frac{u^2}{a^2},\]we find that
		\begin{align*}\left|S(u)\right|&\leq \frac{2\sqrt{2}}{\pi}\frac{u(u-a)}{a^2}\sqrt{t}+\frac{8\sqrt{2}u}{\sqrt{t}}+\frac{t(u-a)}{2\pi a^2}+3.
		\end{align*}
		Using Riemann-Stieltjes integration, we have
		\begin{align*}
			\sum_{a<n\leq b}\frac{1}{n^{ 1+it}}&=\int_a^b \frac{1}{u}dS(u)=\frac{S(b)}{b}+\int_a^b \frac{S(u)}{u^2}du.
		\end{align*}
		Now,
		\begin{align*}
			\left|\frac{S(b)}{b}\right|&\leq \frac{|S(b)|}{b}\leq \frac{2\sqrt{2}(b-a)}{\pi a^2}\sqrt{t} +\frac{8\sqrt{2}}{\sqrt{t}}+\frac{t(b-a)}{2\pi a^2b}+\frac{3}{b}.
		\end{align*}
		On the other hand,
		\begin{align*}
			&\left|\int_a^b \frac{S(u)}{u^2}du\right| \leq \int_a^b \frac{|S(u)|}{u^2}du\\&\leq\frac{2\sqrt{2}}{\pi a^2}\sqrt{t} \int_a^b \frac{u-a}{u}du+\frac{8\sqrt{2}}{\sqrt{t}}\int_a^b \frac{1}{u}du+\frac{t}{2\pi a^2}\int_a^b\frac{u-a}{u^2}du+3\int_a^b\frac{1}{u^2}du\\
			&=\frac{2\sqrt{2}(b-a)}{\pi a^2}\sqrt{t} -\frac{2\sqrt{2}}{\pi a }\sqrt{t} \log\frac{b}{a}+\frac{8\sqrt{2}}{\sqrt{t}}\log\frac{b}{a}+\frac{t}{2\pi a^2}\log\frac{b}{a}+\left(3-\frac{t}{2\pi a}\right)\frac{(b-a)}{ab}.
		\end{align*}
		It follows that
		\begin{align*}
			\left|\sum_{a<n\leq b}\frac{1}{n^{ 1+it}}\right|&\leq\left|\frac{S(b)}{b}\right|+\left|\int_a^b \frac{S(u)}{u^2}du\right|\\
			&\leq \frac{4\sqrt{2}(b-a)}{\pi a^2}\sqrt{t} -\frac{2\sqrt{2}}{\pi a }\sqrt{t} \log\frac{b}{a}+\frac{8\sqrt{2}}{\sqrt{t}}\log\frac{b}{a}+\frac{8\sqrt{2}}{\sqrt{t}}+\frac{t}{2\pi a^2}\log\frac{b}{a}+\frac{3}{a}.
		\end{align*}
	\end{proof}

	Now we can use    \eqref{240114_1} and Proposition \ref{240113_2} to give a better  estimate of $|\zeta(1+it)|$.
	Given $x\geq 1$, let $x_0$ be such that $1\leq x_0\leq x$, and let $k$ be the positive integer such that
	\[\frac{x}{2^k}\leq x_0<\frac{x}{2^{k-1}}.\]
	This implies that
	\begin{equation}\label{20241026_2}\frac{2^k}{x}<\frac{2}{x_0},\hspace{1cm}
		\text{and}\hspace{1cm}
		k< \frac{\log x-\log x_0+\log 2}{\log 2}.\end{equation}
	Split the set of  integers $n$ with $  n\leq x$ into $k+1$ sets $S_0$, $S_1$, $S_2, \ldots, S_k$, where
	\[S_0=\left\{n\,|\, n\leq x_0\right\},\]and 
	for $1\leq j\leq k$, \[ S_j=\left\{n\,\left|\,\frac{x}{2^j}< n\leq \frac{x}{2^{j-1}}\right.\right\}.\]
	For the sum over $S_0$, the crude estimate \eqref{20250529_2} and   Lemma \ref{240112_4} give
	\[\left|\sum_{n\leq x_0}\frac{1}{n^{1+it}}\right|\leq\sum_{n\leq x_0}\frac{1}{n}\leq  
	\log x_0+\gamma+\frac{1}{x_0}.\]
	For the sum over $n$ in $S_j$, $1\leq j\leq k$,
	Proposition \ref{240113_2} gives
	\begin{align*}
		\left|\sum_{\frac{x}{2^j}<n\leq \frac{x}{2^{j-1}}}\frac{1}{n^{1+it}}\right|
		&\leq\frac{2^{j+1}\sqrt{2} }{\pi x}\sqrt{t} \left(2-\log 2\right)+\frac{8\sqrt{2}}{\sqrt{t}}\left(1+\log 2\right) +\frac{2^{2j-1}t}{ \pi x^2}\log 2+\frac{3\times 2^j}{x}.
	\end{align*}
	Therefore,
	\begin{align*}
		\left|\sum_{ x_0<n\leq x}\frac{1}{n^{1+it}}\right|&\leq \sum_{j=1}^k\left\{\frac{2^{j+1}\sqrt{2} }{\pi x}\sqrt{t} \left(2-\log 2\right)+\frac{8\sqrt{2}}{\sqrt{t}}\left(1+\log 2\right)+\frac{2^{2j-1}t}{ \pi x^2}\log 2+\frac{3\times 2^j}{x}\right\}\\
		&=\frac{4(2^{k}-1)\sqrt{2} }{\pi x}\sqrt{t} \left(2-\log 2\right)+\frac{8\sqrt{2}k}{\sqrt{t}}\left(1+\log 2\right) \\&\quad+\frac{ t}{ \pi x^2}\times\frac{2(4^k-1)}{3}\log 2+\frac{6\times (2^k-1)  }{x}.\end{align*}
	Using \eqref{20241026_2}, we have 
	\begin{align*}
		\left|\sum_{ x_0<n\leq x}\frac{1}{n^{1+it}}\right|&\leq    \frac{8\sqrt{2}}{\pi x_0}\sqrt{t} \left(2-\log 2\right) +\frac{8t}{3\pi x_0^2}\log 2 +\frac{12}{x_0}-\frac{6}{x}\\
		&\quad +\frac{8\sqrt{2}\left(1+\log 2\right)  }{\sqrt{t}\log 2} (\log x-\log x_0+\log 2).
	\end{align*}
	It follows from \eqref{240114_1} that
	\begin{equation}\label{240115_1}
		\begin{split}
			|\zeta(1+it)|&\leq \left|\sum_{  n\leq x_0}\frac{1}{n^{1+it}}\right|+\left|\sum_{ x_0<n\leq x}\frac{1}{n^{1+it}}\right|+\frac{1}{t}+\frac{1}{2x}+\frac{t^2+5t+4}{32x^2} \\
			&\leq\log x_0+\gamma+ \frac{8\sqrt{2}}{\pi x_0}\sqrt{t} \left(2-\log 2\right)+\frac{13}{x_0}+\frac{8t}{3\pi x_0^2}\log 2\\
			&\quad +\frac{8\sqrt{2}\left(1+\log 2\right)  }{\sqrt{t}\log 2}(\log x-\log x_0+\log 2)+\frac{1}{t} +\frac{t^2+5t+4}{32x^2}.
	\end{split}\end{equation}
	For any fixed $t$, we can use elementary calculus to find $x_0$ and $x$ that minimize the expression on the right. More precisely,  
	let \[e_0=\frac{8\sqrt{2}\left(1+\log 2\right)  }{ \log 2},\quad e_1=\frac{8\sqrt{2}}{\pi  }  \left(2-\log 2\right), \quad e_2=\frac{8 }{3\pi  }\log 2, \]
	\begin{equation}\label{20250914_3}E_0(t)=1-\frac{e_0}{\sqrt{t}},\quad E_1(t)=e_1\sqrt{t} +13, \quad E_2(t)=e_2t,\end{equation}
	\begin{equation}\label{20250914_4}G_0(t)=\frac{e_0}{\sqrt{t}},\quad G_2(t)=\frac{t^2+5t+4}{ 32},\end{equation}
	\begin{equation}\label{20250612_5}Q_0(t)=\gamma+\frac{e_0\log 2}{\sqrt{t}}+\frac{1}{t}.\end{equation}
	Then \eqref{240115_1} says that
	\begin{equation}\label{20250914_6}|\zeta(1+it)|\leq  E_0(t)\log x_0+\frac{E_1(t)}{x_0}+\frac{E_2(t)}{x_0^2}+G_0(t)\log x+\frac{G_2(t)}{x^2}+Q_0(t)\end{equation}for any $x_0$ and $x$ satisying $1\leq x_0\leq x$. 
	Notice that $  E_1(t), E_2(t), G_0(t), G_2(t), Q_0(t)$ are positive   for  any $t>0$, while $E_0(t)>0$ if and only if
	\[t>e_0^2=763.75.\] For fixed $t>0$, consider the functions
	\begin{subequations}
	\begin{equation} h_E(y) =E_0(t) \log y+\frac{E_1(t)}{y}+\frac{E_2(t)}{y^2},\hspace{1cm} y\geq 1,\label{20250612_2}\end{equation}\begin{equation}h_G(y) =G_0(t)\log y+\frac{G_2(t)}{y^2},\hspace{1cm} y\geq 1. \end{equation}\end{subequations}Note that the functions $h_E(y)$ and $h_G(y)$ are functions depending on $t$. For fixed $t$,
	the derivative of $h_E(y)$ with respect to $y$ is
	\[h_E'(y)=\frac{E_0(t)}{y}-\frac{E_1(t)}{y^2}-\frac{2E_2(t)}{y^3}=\frac{E_0(t)y^2-E_1(t)y-2E_2(t)}{y^3}.\] 
Its sign is determined by the sign of the quadratic function 
\begin{equation}\label{20250914_2}Q_E(y)= E_0(t)y^2-E_1(t)y-2E_2(t).\end{equation}
When $t\leq 763.75$, $E_0(t)\leq 0$ and so $h_E'(y)\leq 0$ for all $y>0$. This implies that  when $t\leq 763.75$, $h_E(y)$ is  a strictly decreasing function of $y$. When $t>763.75$, $E_0(t)>0$. The discriminant of the quadratic function $Q(y)$ \eqref{20250914_2}  is\[\Delta_E(t)=E_1(t)^2+8E_0(t)E_2(t),\]which is positive. Therefore, when $t>763.75$, $Q_E(y)$ has two distinct real roots $y_1$ and $y_2$. Assume that $y_1<y_2$. Then we must have $y_1<0<y_2$.   For $0<y<y_2$, $Q_E(y)<0$. For $y>y_2$, $Q_E(y)>0$. Hence, the function $h_E(y)$ has minimum value at the point
	\begin{equation}\label{20250611_5}y=y_E(t)=y_2=\frac{E_1(t)+\sqrt{E_1^2(t)+8E_0(t)E_2(t)}}{2E_0(t)}.\end{equation}For the function $h_G(y)$, its derivative with respect to $y$ is
	\[h_G'(y)=\frac{G_0(t)}{y}-\frac{2G_2(t)}{y^3}=\frac{G_0(t)y^2-2G_2(t)}{y^3}.\]Its sign is determined by the quadratic function $Q_G(y)=G_0(t)y^2-2G_2(t)$. Since $G_0(t)$ is positive for all 
	  $t>0$, a similar argument shows that  the function $h_G(y)$ has minimum value at  the point
	\begin{equation}\label{20250611_6}y=y_G(t)=\sqrt{\frac{2G_2(t)}{G_0(t)}}.\end{equation}For fixed $t$, the expression \eqref{20250914_6} is minimized when we take $x_0=y_E(t)$ and $x=y_G(t)$. Since we must have $x_0\leq x$, we want to give some simple bounds to $y_E(t)$ and $y_G(t)$ to determine a value $t_0$ so that $y_E(t)\leq y_G(t)$ when $t\geq t_0$. 

	When $t>763.75$, we obtain from \eqref{20250611_5} and \eqref{20250914_3} that
	\[y_E(t)\leq \frac{\di \left(e_1+\frac{13}{\sqrt{t}} +\sqrt{ e_1^2+8e_2+\frac{26e_1}{\sqrt{t}}+\frac{169}{t}}\right) \sqrt{t}}{\di 2\left(1-\frac{e_0}{\sqrt{t}}\right)}.\]
	 On the other hand,  \eqref{20250611_6} and \eqref{20250914_4} give
	 \[y_G(t)\geq \frac{1}{\sqrt{16e_0}}t^{\frac{5}{4}}.\]
When $t>763.75$, the function
\[ \xi_E(t)=\frac{\di \left(e_1+\frac{13}{\sqrt{t}} +\sqrt{ e_1^2+8e_2+\frac{26e_1}{\sqrt{t}}+\frac{169}{t}}\right) }{\di 2\left(1-\frac{e_0}{\sqrt{t}}\right)}\] is decreasing in $t$, while the function
\[\xi_{G}(t)=\frac{1}{\sqrt{16e_0}}t^{\frac{3}{4}} \]is increasing in $t$. Since  $\xi_E(2000)\leq\xi_G(2000)$, we find that when $t\geq 2000$, $\xi_E(t)\leq\xi_G(t)$.
Hence,
	when $t\geq 2000$, we find that \[y_E(t)\leq \frac{\di \left(e_1+\frac{13}{\sqrt{t}} +\sqrt{ e_1^2+8e_2+\frac{26e_1}{\sqrt{t}}+\frac{169}{t}}\right) \sqrt{t}}{\di 2\left(1-\frac{e_0}{\sqrt{t}}\right)}\leq  \frac{1}{\sqrt{16e_0}}t^{\frac{5}{4}}\leq  y_G(t).\] Therefore, when $t\geq 2000$, we have
	 \[|\zeta(1+it)|\leq h_E(y_E(t))+h_G(y_G(t))+Q_0(t).\]
	When $t$ is large, 
	\begin{equation}\label{20250611_1}y_E(t)\sim \lambda_1\sqrt{t}, \hspace{1cm} y_G(t)=O\left( t^{\frac{5}{4}}\right),\end{equation}
	where
	\begin{equation}\label{20250612_3}\lambda_1=\frac{e_1+\sqrt{e_1^2+8e_2}}{2}=4.9443.\end{equation}
	Hence,
	\begin{equation}\label{20250612_4}h_E(y_E(t))\sim \frac{1}{2}\log t+\lambda_2, \quad h_G(y_G(t))=O\left(t^{-\frac{1}{2}}\log t\right), \end{equation}
	where 
	\[\lambda_2=\log\lambda_1+\frac{e_1}{\lambda_1}+\frac{e_2}{\lambda_1^2}=2.5742.\]Therefore, when $t$ is large,
	\[  h_E(y_E(t))+h_G(y_G(t))+Q_0(t)\sim \frac{1}{2}\log t+\lambda_2+\gamma=  \frac{1}{2}\log t+3.1514.\]
	These show that using exponential sums with second order derivatives only, we cannot achieve a bound for $|\zeta(1+it)|$ that is better than $ \frac{1}{2}\log t$. 
	
	Now we turn to   bounds of the form 
	\begin{equation}\label{20250611_3}|\zeta(1+it)|\leq v\log t\hspace{1cm}\text{when}\;t\geq t_0.\end{equation}
To achieve this, we take
	\begin{equation} \label{20241123_1} x=   t^u, \quad x_0= \beta t^v, \end{equation} where   $\beta$, $u$, $v$ are positive constants satisfying
	\[v\leq u,\hspace{1cm} \beta\leq 1.\]
	These conditions ensure that $x_0\leq x$. 
	Substituting into \eqref{240115_1},   we find that
	\begin{equation}\label{20241026_3}|\zeta(1+it)|\leq h_E(\beta t^v)+h_G(t^u)+Q_0(t) =v\log t+A+\omega(t),\end{equation}
	where 
	\begin{align*}
		A&=\log\beta +\gamma,\\
		\omega(t)&= \frac{8\sqrt{2}}{\pi \beta}t^{\frac{1}{2}-v} \left(2-\log 2\right) +\frac{13}{\beta t^v} +\frac{8t^{1-2v}}{3\pi \beta^2}\log 2+\frac{1}{t} +\frac{t^2+5t+4}{32 t^{2u}} \\
		&\quad +\frac{8\sqrt{2}\left(1+\log 2\right)  }{ \log 2} t^{-\frac{1}{2}}\bigl[(u-v)\log t -\log\beta+\log 2\bigr].
	\end{align*}

	To make $v\log t$ the leading term in \eqref{20241026_3}, we must have 
	\[u\geq 1, \quad v\geq\frac{1}{2}.\]
	For fixed $u$ and $v$ with $\frac{1}{2}\leq v\leq u$, the function $\omega(t)$ is   decreasing on $[e^2, \infty)$. If we can find $\beta\in (0,1]$ and $t_0\geq e^2$ so that
	\[A+\omega(t_0) = 0,\]then 
	\[A+\omega(t)  \leq 0\hspace{1cm}\text{for all}\;t\geq t_0.\] This will imply that
	\[|\zeta(1+it)|\leq v\log t\hspace{1cm}\text{when}\;t\geq t_0.\]
	
	When $v=\frac{1}{2}$ and $u>1$,
	\[\lim_{t\to\infty}(A+\omega(t))= \gamma+\log\beta+\frac{8\sqrt{2}}{\pi \beta} \left(2-\log 2\right)  +\frac{8 }{3\pi \beta^2}\log 2=\gamma+\log\beta+\frac{e_1}{\beta}+\frac{e_2}{\beta^2}.\]
	The function 
	\[h_C(\beta)=\gamma+\log\beta+\frac{e_1}{\beta}+\frac{e_2}{\beta^2}, \quad \beta>0\]
	has a minimum at 
	\begin{equation}\label{20250914_9}\beta=\frac{e_1+\sqrt{e_1^2+8e_2}}{2}=\lambda_1=4.9443,\end{equation}with minimum value \[h_C(4.9443)=\gamma+\log\lambda_1+\frac{e_1}{\lambda_1}+\frac{e_2}{\lambda_1^2}= \gamma+2.5742=3.1514.\]
	Therefore, using this method, we cannot achieve a result better than
	\[|\zeta(1+it)|\leq \frac{1}{2}\log t+3.1514,\]agreeing with our earlier asymptotic analysis.

Now consider the case where $  v>\frac{1}{2}$. When $u$ is also greater than 1, we have 
\[\lim_{t\to\infty}\omega(t)=0.\]
Hence, if $\beta<e^{-\gamma}=0.5615$, there   exists $t_0>e^2$ such that
\[A+\omega(t_0)=0.\]
Instead of fixing $v>\frac{1}{2}$ and looking for  $t_0$, we can fix $t_0$ and look for $u$, $v$ and $\beta$ that can achieve 
  \eqref{20250611_3}. As discussed above,  we want to have $A+\omega(t_0)=0$. By \eqref {20241026_3}, this is achieved if
  \begin{equation}\label{20250614_2} h_E(\beta t_0^v)+h_G(t_0^u)+Q_0(t_0) =v\log t_0.\end{equation}
  To make $v$ the smallest possible for the given $t_0$, we need to choose $\beta$, $u$ and $v$ that will minimize the left hand side. From our earlier analysis, when $t_0\geq 2000$, the minimum of the left hand side is achieved when $\beta t_0^v=y_E(t_0)$ and $t_0^u=y_G(t_0)$, with $y_E(t_0)\leq y_G(t_0)$. Hence,  for given $t_0\geq 2000$, 
  we first compute $y_E(t_0)$ and $y_G(t_0)$ from \eqref{20250611_5} and \eqref{20250611_6} respectively.   By \eqref{20250614_2}, we should let $v$ be the positive number such that
\begin{equation}\label{20250611_8}v=\frac{h_E(y_E(t_0))+h_G(y_G(t_0))+Q_0(t_0)}{\log t_0}.\end{equation}
Then   $u$ and $\beta$ should be defined so that
\begin{equation}\label{20250611_9}\beta t_0^v=y_E(t_0)\hspace{1cm}\text{and}\hspace{1cm} t_0^u=y_G(t_0).\end{equation}
For $v$, $\beta$ and $u$ defined in this way, we must have  
$
A+\omega(t_0) =0$. 
From what we have discussed above, this implies that 
\begin{equation}\label{20250614_7}|\zeta(1+it)|\leq v\log t\hspace{1cm}\text{when}\;t\geq t_0.\end{equation}
For a fixed $t_0\geq 2000$, the $v$ obtained in this way is the minimum possible $v$   such that  \eqref{20250614_7} holds.

\begin{table}[ht]\caption{The values of  $\beta$, $v$ and $u$ satisying \eqref{20250611_8} and \eqref{20250611_9} for given $t_0$.\label{tab1}}
\begin{tabular}{|c|c|c|c|}
\hline
 	$t_0$	&	$\beta$ &	$v$ &	$u$	\\\hline
$	10^{5}	$&$	0.1474	$&$	0.8134	$&$	0.9854	$\\\hline
$	10^{6}	$&$	0.1796	$&$	0.7421	$&$	1.0295	$\\\hline
$	10^{7}	$&$	0.1978	$&$	0.7003	$&$	1.0610	$\\\hline
$	10^{8}	$&$	0.2061	$&$	0.6726	$&$	1.0847	$\\\hline
$	10^{9}	$&$	0.2095	$&$	0.6526	$&$	1.1030	$\\\hline
$	10^{10}	$&$	0.2108	$&$	0.6370	$&$	1.1177	$\\\hline
$	10^{11}	$&$	0.2113	$&$	0.6245	$&$	1.1297	$\\\hline
$	10^{12}	$&$	0.2115	$&$	0.6141	$&$	1.1398	$\\\hline
$	10^{13}	$&$	0.2115	$&$	0.6053	$&$	1.1482	$\\\hline
$	10^{14}	$&$	0.2116	$&$	0.5978	$&$	1.1555	$\\\hline
$	10^{15}	$&$	0.2116	$&$	0.5912	$&$	1.1618	$\\\hline
$	10^{20}	$&$	0.2116	$&$	0.5684	$&$	1.1839	$\\\hline
$	10^{30}	$&$	0.2116	$&$	0.5456	$&$	1.2059	$\\\hline
$	10^{40}	$&$	0.2116	$&$	0.5342	$&$	1.2169	$\\\hline
$	10^{50}	$&$	0.2116	$&$	0.5274	$&$	1.2235	$\\\hline
$	10^{60}	$&$	0.2116	$&$	0.5228	$&$	1.2280	$\\\hline
$	10^{70}	$&$	0.2116	$&$	0.5196	$&$	1.2311	$\\\hline
$	10^{80}	$&$	0.2116	$&$	0.5171	$&$	1.2335	$\\\hline
$	10^{90}	$&$	0.2116	$&$	0.5152	$&$	1.2353	$\\\hline
$	10^{100}	$&$	0.2116	$&$	0.5137	$&$	1.2368	$\\\hline
$	10^{200}	$&$	0.2116	$&$	0.5068	$&$	1.2434	$\\\hline
$\hspace{0.5cm}	10^{300}	\hspace{0.5cm}$&$	\hspace{0.5cm}0.2116\hspace{0.5cm}	$&$	\hspace{0.5cm}0.5046\hspace{0.5cm}	$&$	\hspace{0.5cm}1.2456\hspace{0.5cm}	$\\\hline

\end{tabular}

\end{table}

In Table \ref{tab1}, we list down some values of $\beta$, $v$ and $u$ computed using \eqref{20250611_8} and \eqref{20250611_9} with given $t_0$. We can see the trend that as $t_0$ gets large,   $v$ approaches the limit $\frac{1}{2}=0.5$, while $u$ approaches the limit $\frac{5}{4}=1.25$. As for $\beta$, the equations \eqref{20250611_9} and \eqref{20250611_8} say  that
\begin{equation}\label{20250612_6}-\log\beta= \left(h_E(y_E(t_0))-\log y_E(t_0)\right)+h_G(y_G(t_0))+Q_0(t_0).\end{equation}
By  \eqref{20250612_2} and \eqref{20250611_1}, we find that when $t$  is large,
\[h_E(y_E(t_0))-\log y_E(t_0) \sim \frac{E_1(t_0)}{y_E(t_0)}+\frac{E_2(t_0)}{y_E(t_0)^2} \sim \frac{e_1}{\lambda_1}+\frac{e_2}{\lambda_1^2},\]
where $\lambda_1$ is given by \eqref{20250612_3}. From \eqref{20250612_4}, we  have $h_G(y_G(t_0))\to 0$ when $t_0\to\infty$. From the definition of $Q_0(t)$ \eqref{20250612_5}, $Q_0(t_0)\to \gamma$ when $t_0\to\infty$. Hence,   when $t_0$ is large, \eqref{20250612_6} shows that $\beta$ should approach the limiting value
\[\exp\left( -\frac{e_1}{\lambda_1}-\frac{e_2}{\lambda_1^2}-\gamma\right)=0.2116.\]This is indeed the case as shown in Table \ref{tab1}.
	
Once we have shown that the inequality $|\zeta(1+it)|\leq v\log t$ holds for $t\geq t_0$, we can extend this to $t\geq t_1$ for some $0<t_1<t_0$ by   numerically computing $\zeta(1+it)$  for $t_1\leq t<t_0$. For example, we have shown that
\[|\zeta(1+it)|\leq 0.7421\log t \hspace{1cm}\text{when}\;t\geq 10^6.\]
When $e\leq t<10^6$, numerical calculations of $|\zeta(1+it)|$ using the code in Section \ref{sec2} show that we still have
	\[|\zeta(1+it)|\leq 0.7421\log t.\]
	Hence, we conclude that \[|\zeta(1+it)|\leq 0.7421\log t \hspace{1cm}\text{when}\;t\geq e.\]
	This  implies the result  in  \cite{Trudgian_2014} which says that
	\[|\zeta(1+it)|\leq 0.75\log t\hspace{1cm}\text{when}\;t\geq e.\]
	From Table \ref{tab1}, we find that 
	\[|\zeta(1+it)\leq 0.6370\log t\hspace{1cm}\text{when}\;t\geq 10^{10}.\]  If we can numerically compute $\zeta(1+it)$ for $t$ up to $10^{10}$, we can then verify that
	\[|\zeta(1+it)|\leq 0.6443\log t \hspace{1cm}\text{when}\;t\geq e.\]However, ordinary computers cannot handle  computations based on the simple algorithm given in Section \ref{sec2} up to this value of $t$. To go around this, we use the Riemann--Siegel formula.

	\bigskip
	\section{Bounds of $|\zeta(1+it)|$ Using the Riemann--Siegel Formula}\label{sec4}
	When $t$  is large, Theorem \ref{20241022_6} does not give an efficient way  to compute $\zeta(1+it)$. For  a more efficient way, we can use the Riemann--Siegel formula.  
	
	 In 1932, Siegel presented in his paper \cite{Siegel1932}  an unpublished result of Riemann and gave derivations to the formula that is now known as
	the  Riemann--Siegel formula.  In this section, we are mainly concerned with using  the Riemann--Siegel formula  to obtain a better bound for $|\zeta(1+it)|$. For efficient computations of $\zeta(s)$ using the Riemann--Siegel formula, one  can refer to the work \cite{AriasdeReyna_2011}.
	
	The starting point of the Riemann--Siegel formula is the following representation of the Riemann zeta function given by Riemann.
	\begin{theorem} [Riemann]
		For $s\in \mathbb{C}$,
		\begin{equation}\label{20241026_7}\zeta(s)=\mathcal{R}(s)+\chi(s)\overline{\mathcal{R}}(1-s),\end{equation}where
		\[\mathcal{R}(s)=\int_{0\swarrow 1} \frac{w^{-s}e^{i\pi w^2}}{e^{i\pi w}-e^{-i\pi w}}dw,\hspace{1cm}\overline{\mathcal{R}}(s)=\overline{\mathcal{R}(\overline{s})},\]
		and
		\begin{equation}\label{231217_2}\chi(s)=\pi^{s-\frac{1}{2}}\frac{\Gamma\left(\frac{1-s}{2}\right)}{\Gamma\left(\frac{s}{2}\right)}.\end{equation}
		The integration contour $0\swarrow 1$ in the definition of $\mathcal{R}(s)$ is the line $a+te^{-\frac{3\pi i}{4}}$, $t\in\mathbb{R}$ that passes through a point $a$ between 0 and 1, with southwest direction as indicated by the arrow.
	\end{theorem}
	
	The main ingredient of the Riemann--Siegel formula is the  asymptotic expansion for the function
	\[\mathcal{R}(s)=\int_{0\swarrow 1} \frac{w^{-s}e^{i\pi w^2}}{e^{i\pi w}-e^{-i\pi w}}dw.\] 
	 
		Let $\tau$ and $z$ be   complex numbers. For a real number $\sigma$, consider the function 
		\begin{equation}\label{20241103_2}g(\tau,z)=\exp\left\{-\left(\sigma+\frac{i}{8\tau^2}\right)\log(1+2i\tau z)-\frac{z}{4\tau}+i\frac{z^2}{4}\right\}.\end{equation}
		When $|2\tau z|<1$, Taylor series expansion about the point $z=0$ gives
		 \begin{align*}-\left(\sigma+\frac{i}{8\tau^2}\right)\log(1+2i\tau z)-\frac{z}{4\tau}+i\frac{z^2}{4}
			= \sum_{k=1}^{\infty}(-1)^{k}\left(\sigma\frac{(2i)^k}{k}z^k+\frac{(2i)^{k-1}}{k+2}z^{k+2}\right)\tau^k.
		\end{align*}Hence, $g(\tau, z)$ has an expansion of the form
		\begin{equation}\label{20241026_10}g(\tau, z)=\sum_{k=0}^{\infty} P_k(z)\tau^k,\end{equation}where  $P_k(z)$ is a polynomial in $z$ and $\sigma$.  In particular, $P_0(z)$ and $P_1(z)$ are given respectively by 
	\begin{equation}\label{20241027_5}P_0(z)=1, \hspace{1cm}P_1(z)= -\frac{1}{3}z^3-2i\sigma z.\end{equation}
	The following expansion of $\mathcal{R}(s)$ given in \cite{AriasdeReyna_2011}  is often attributed to Lehmer \cite{Lehmer56}. For a proof, see  \cite{AriasdeReyna_2011}. 
	\begin{theorem}
		Given $s=\sigma+it$ with $t>0$, let
		\begin{gather*}
			a =\sqrt{\frac{t}{2\pi}},\quad 
			N =\lfloor a\rfloor,\quad 
			p=1-2a+2N,\quad\tau=\frac{1}{\sqrt{8t}},\end{gather*}
			\begin{equation}\label{20250605_1}
			U =\exp\left\{-i\left[\frac{t}{2}\log\frac{t}{2\pi}-\frac{t}{2}-\frac{\pi}{8}\right]\right\}.
		\end{equation}Then $-1\leq p\leq 1$.
		For any nonnegative integer $K$,  the function 
		\[\mathcal{R}(s)=\int_{0\swarrow 1} \frac{w^{-s}e^{i\pi w^2}}{e^{i\pi w}-e^{-i\pi w}}dw\]
		has an expansion of the form
		\begin{equation}\label{20241026_4}\mathcal{R}(s)=\sum_{n=1}^N\frac{1}{n^s}+(-1)^{N-1}Ua^{-\sigma}\left\{\sum_{k=0}^K\frac{C_k(p)}{a^k}+RS_K(p)\right\},\end{equation}where
		\begin{align}
			C_k(p)&=\frac{e^{-\frac{i\pi}{8}}}{4}\frac{1}{(4\sqrt{\pi})^k}\int_{\searrow ip}\frac{e^{-\frac{i\pi}{2}(v-ip)^2}}{\cosh\frac{\pi}{2}v}P_k\left( \sqrt{\pi}(v-ip)\right)dv,\label{20241026_11}\\
			RS_K(p)&=\frac{e^{-\frac{i\pi}{8}}}{4}\int_{\searrow ip}\frac{e^{-\frac{i\pi}{2}(v-ip)^2}}{\cosh\frac{\pi}{2}v}Rg_K\left(\tau, \sqrt{\pi}(v-ip)\right)dv,\label{20250606_1}
		\end{align}
		$g(\tau, z)$ and $P_k(z)$ are defined in \eqref{20241103_2} and \eqref{20241026_10},  
		\[Rg_K(\tau, z)=g(\tau, z)-\sum_{k=0}^KP_k(z)\tau^k,\] and the integration path $\searrow ip$ is the line $ip+te^{-\frac{i\pi}{4}}, t\in\mb{R}$ that passes through the point $ip$ and pointing to the southeast direction as indicated by the arrow.
	\end{theorem}

	By definition \eqref{20250605_1}, $|U|=1$. Hence, \eqref{20241026_4} gives the following bound of $\mathcal{R}(\sigma+it)$:
	\begin{equation}\label{20241026_8}|\mathcal{R}(\sigma+it)|\leq \left|\sum_{n=1}^N\frac{1}{n^{\sigma+it}}\right|+\left(\frac{2\pi}{t}\right)^{\frac{\sigma}{2}}\left\{\sum_{k=0}^K |C_k(p)|\left(\frac{2\pi}{t}\right)^{\frac{k}{2}}+|RS_K(p)|\right\}.\end{equation}
	Since $P_0(z)=1$, $C_0(p)$ is independent of $\sigma$, but $C_k(p)$ depends on $\sigma$ for all $k\geq 1$. Similarly, $RS_K(p)$ depends on $\sigma$ for all $K\geq 0$. 
	
	Specializing to $\sigma=1$, \eqref{20241026_7} gives
	\[\zeta(1+it)=\mathcal{R}(1+it)+\chi(1+it)\overline{\mathcal{R}}(-it).\]As in \cite{Patel2022}, we take $K=1$ in \eqref{20241026_8} and obtain
	\begin{equation}\label{20241026_9}
		|\zeta(1+it)| 
		\leq \left|\sum_{n\leq \sqrt{t/(2\pi)}}\frac{1}{n^{1+it}}\right| +|\chi(1+it)|\left| \sum_{n\leq \sqrt{t/(2\pi)}}\frac{1}{n^{-it}}\right|+\varkappa(t),\end{equation}
	where
	\begin{align*}\varkappa(t)&=\sqrt{\frac{2\pi}{t}}\left(b_0+b_1(1)\sqrt{\frac{2\pi}{t}}+\frac{   c(1)}{ t}\right)+|\chi(1+it)|\left(b_0+b_1(0)\sqrt{\frac{2\pi}{t}}+\frac{  c(0)}{ t}\right),\end{align*}with
	\begin{align}b_0&=\max_{-1\leq p\leq 1}|C_0(p)|,\label{20250606_2}\\
		b_1(\sigma)&=\max_{-1\leq p\leq 1}|C_1(p)|,\label{20250606_3}\\
		c(\sigma)&=\max_{-1\leq p\leq 1}|tRS_1(p)|.\label{20250606_4}
	\end{align}In the following, we find upper bounds for $|\chi(1+it)|$, $b_0$, as well as $b_1(\sigma)$ and $c(\sigma)$ when $\sigma=0$ and $\sigma=1$, and compare our results to  \cite{Patel2022}.

	We start with $|\chi(1+it)|$. By Euler's reflection formula
		and Legendre's duplication formula, we find that
		\begin{align*}
			\chi(s)=\pi^{s-\frac{1}{2}}\frac{\Gamma\left(\frac{1-s}{2}\right)}{\Gamma\left(\frac{s}{2}\right)}=\pi^{s-\frac{1}{2}}\frac{\Gamma\left(\frac{1-s}{2}\right)\Gamma\left(\frac{1+s}{2}\right)}{\Gamma\left(\frac{s}{2}\right)\Gamma\left(\frac{1+s}{2}\right)}=2^{s-1}\pi^s\frac{1}{\sin   \frac{\pi(1-s)}{2} }\frac{1}{\Gamma(s)}.
		\end{align*}
	Hence, when $t>0$, 
		\begin{equation} \label{231217_5}\begin{split}
				\chi(1+it)&=-2^{it}\pi^{1+it}\frac{1}{\Gamma\left(1+it\right)}\frac{1}{\sin \frac{\pi i t}{2}}=2^{1+it}i\pi^{1+it}\frac{1}{\Gamma\left(1+it\right)}\frac{1}{\left(e^{\frac{\pi t}{2}}-e^{-\frac{\pi t}{2}}\right)}.  
		\end{split}\end{equation}
By Theorem 1.4.2 in \cite{Andrews}, we find that when $\text{Re}\,s>0$,
		\begin{align*}
			\log\Gamma(s)=\left(s-\frac{1}{2}\right)\log s-s+\frac{1}{2}\log(2\pi)+\frac{1}{12s}-\frac{1}{2}\int_0^{\infty}\frac{B_2(\{x\})}{(x+s)^2}dx,
		\end{align*}where $B_2(x)=x^2-x+\frac{1}{6}$ is the second Bernoulli polynomial. As explained in Section \ref{sec2}, we can replace $B_2(x)$ with the function $h(x)=x^2-x+\frac{1}{8}$ and obtain 		
		\begin{equation}\label{231217_1}\log\Gamma(s)=\left(s-\frac{1}{2}\right)\log s-s+\frac{1}{2}\log(2\pi)+\frac{1}{16s}-\frac{1}{2}\int_0^{\infty}\frac{h(\{x\})}{(x+s)^2}dx.\end{equation}
		In the following theorem, we use this formula to obtain an upper bound of $|\chi(1+it)|$.
	\begin{theorem}\label{231217_7}
		For $t>0$, we have
		\begin{equation}\label{20250606_5}|\chi(1+it)|\leq \sqrt{\frac{2\pi}{t}}\exp\left(\frac{\pi}{32t}-\frac{1}{24t^2}+\frac{5}{24t^4}\right)\frac{1}{1-e^{-\pi t}}.\end{equation}
	\end{theorem}
	\begin{proof}	 
	When $t>0$, \eqref{231217_5} gives
		\begin{equation}\label{20250605_2}
			|\chi(1+it)|=2\pi\left|\frac{1}{\Gamma\left(1+it\right)}\right|\frac{e^{-\frac{\pi t}{2}}}{\left(1-e^{- \pi t }\right)}.
		\end{equation} When $t>0$,
		\[\log (1+it)=\frac{1}{2}\log(1+t^2)+i\tan^{-1}t.\] 
	Hence \eqref{231217_1} gives
		\begin{equation}\label{20241103_1}\begin{split}
		\log\frac{1}{|\Gamma(1+it)|}&=	-\text{Re}\,  \log\Gamma(1+it)\\& =-\frac{1}{4}\log(1+t^2)+t\tan^{-1}t+1-\frac{1}{2}\log(2\pi) -\frac{1}{16(1+t^2)}
				\\&\quad +\frac{1}{2}\text{Re}\,\int_0^{\infty}\frac{h(\{x\}) }{ \left(x+1+it\right)^2}dx.
		\end{split}\end{equation}
	By Taylor's remainder theorem,
	\[\log\left(1+\frac{1}{t^2}\right)\geq \frac{1}{t^2}-\frac{1}{2t^4}.\]Therefore,
	 \[\log(1+t^2) =2\log t+\log\left(1+\frac{1}{t^2}\right)\geq 2\log t+\frac{1}{t^2}-\frac{1}{2t^4}.\]
	 It is easy to see that
	 \[\frac{1}{1+u^2}\geq 1-u^2\hspace{1cm}\text{for all}\;u\in\mathbb{R}.\]
	 Therefore,
		  \begin{align*}
			 \frac{1}{ 1+t^2}= \frac{1}{t^2\left(1+\di\frac{1}{t^2}\right)}\geq\frac{1}{t^2}\left(1-\frac{1}{t^2}\right)\geq \frac{1}{t^2}-\frac{1}{ t^4}.
		\end{align*}
		On the other hand, when $x\geq 0$,
		\[\tan^{-1}x=\int_0^x\frac{1}{1+u^2}du\geq \int_0^x (1-u^2)du=x-\frac{x^3}{3}.\]
		This gives
		\[t\tan^{-1}t+1=t\left(\frac{\pi}{2}-\tan^{-1}\frac{1}{t}\right)+1\leq\frac{\pi t}{2}+\frac{1}{3t^2}.\]
		For the last term in \eqref{20241103_1}, we find that
		\begin{align*}
			\frac{1}{2}\text{Re}\,\int_0^{\infty}\frac{h(\{x\}) }{ \left(x+1+it\right)^2}dx&\leq \frac{1}{2}\left|\int_0^{\infty}\frac{h(\{x\}) }{ \left(x+1+it\right)^2}dx\right|\\&\leq \frac{1}{16}\int_0^{\infty}\frac{1}{(x+1)^2+t^2}dx
		 =\frac{1}{16}\int_1^{\infty}\frac{dx}{x^2+t^2}\\&=\frac{1}{16 t}\left(\frac{\pi}{2}-\tan^{-1}\frac{1}{t}\right) \leq\frac{1}{16}\left(\frac{\pi}{2t}-\frac{1}{t^2}+\frac{1}{3t^4}\right).
		\end{align*}
		Collecting together the estimates, we find that
		\begin{align*}
			\log\left|\frac{1}{\Gamma(1+it)}\right|&\leq  
			 -\frac{1}{2}\log (2\pi t)+\frac{\pi t}{2}+\frac{\pi}{32t}-\frac{1}{24t^2}+\frac{5}{24t^4}.
		\end{align*}
		From this, we conclude that
		\[\left|\frac{1}{\Gamma(1+it)}\right|\leq \frac{1}{\sqrt{2\pi t}}\exp\left(\frac{\pi t}{2}+\frac{\pi}{32t}-\frac{1}{24t^2}+\frac{5}{24t^4} \right). \]Then \eqref{20250606_5} follows from  \eqref{20250605_2}. \qedhere
	\end{proof}
	Our result \eqref{20250606_5} is better than the bound
	\[|\chi(1+it)|\leq \sqrt{\frac{2\pi}{t}}\exp\left(\frac{\pi}{6t}+\frac{5}{3t^2} \right)\]obtained in \cite{Patel2022}.  The reason behind this is we use the alternate expression \eqref{231217_5} to bound $|\chi(1+it)|$ while Patel  \cite{Patel2022} obtained bounds for the terms $|\Gamma(\frac{-it}{2})|$ and $|\Gamma(\frac{1+it}{2})|$  in the numerator and denominator of \eqref{231217_2} separately.

	Next we turn to $b_0$ \eqref{20250606_2}. By Theorem 6.1 in  \cite{AriasdeReyna_2011}, we have the following.
	\begin{theorem}\label{20241103_3}
		For $p\in [-1,1]$, let
		\begin{equation}\label{20241016_13}C_0(p)=\frac{e^{-\frac{i\pi}{8}}}{4} \int_{ \searrow ip}\frac{e^{-\frac{\pi i}{2}(v-ip)^2}}{\cosh\frac{\pi v}{2}}dv.\end{equation}
		Then
		\[b_0=\max_{-1\leq p\leq 1}|C_0(p)|= \frac{1}{2}.\]
		
	\end{theorem}
	\begin{proof}The proof given in \cite{AriasdeReyna_2011} is  for the more general case. Here we give a  straightforward proof.
		As mentioned in \cite{AriasdeReyna_2011}, the integral \eqref{20241016_13}
		 can be evaluated explicitly, and it is given by (see for example \cite{Chandrasekharan}) 
		\begin{equation}\label{20241016_15}C_0(p)=\frac{1}{2\cos\pi p}\left(\exp\left\{\pi i\left(\frac{p^2}{2}+\frac{3}{8}\right)\right\}-i\sqrt{2}\cos\frac{\pi p}{2}\right).\end{equation}
		This is an entire function, and $C_0(-p)=C_0(p)$. Hence, to determine $b_0$, it is sufficient to consider $C_0(p)$ for $0\leq p\leq 1$.
		Figure \ref{fig3} shows the function $|C_0(p)|$ when $p\in [0,1]$.

		\begin{figure}[ht]
			
			\begin{center}  \includegraphics[scale=0.5]{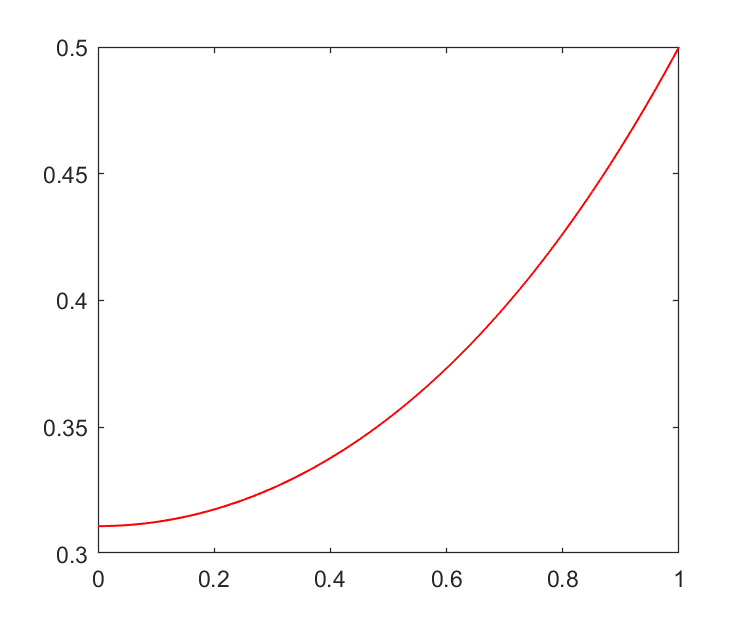}\end{center}
			
			\caption{The function $|C_0(p)|$ when $0\leq p\leq 1$.\label{fig3}}
		\end{figure}
		
		From Figure \ref{fig3}, we find that $|C_0(p)|$ is increasing on $[0,1]$. Since 
		\[C_0(1)=-\frac{1}{2}e^{\frac{7\pi}{8}i},\]
		we find that
		\[b_0=|C_0(1)|=\frac{1}{2}. \qedhere\]
	\end{proof}
	Now we consider $b_1(\sigma)$ \eqref{20250606_3}.
	By \eqref{20241026_11}, we find that
	\begin{equation}\label{20241016_14}
		C_1(p) =\frac{e^{-\frac{i\pi}{8}}}{16\sqrt{\pi}}\int_{ \searrow ip}\frac{e^{-\frac{\pi i}{2}(v-ip)^2}}{\cosh\frac{\pi v}{2}}P_1\left(\sqrt{\pi}(v-ip)\right)dv,
	\end{equation}
	where  
	\begin{align*}
		P_1(z)= -\frac{1}{3}z^3-2i\sigma z
	\end{align*} is the polynomial  given in \eqref{20241027_5}.
	
	\begin{theorem}\label{20241103_4}
		Let
		\[b_1(\sigma)=\max_{-1\leq p\leq 1}|C_1(p)|.\]
		Then 
		\begin{equation}\label{20250606_6}
		b_1(0)= 0.0173\hspace{1cm}\text{and}\hspace{1cm} b_1(1)= 0.0932.\end{equation}
	\end{theorem}
	\begin{proof}
		By definition,
		\begin{equation*}
			 C_1(p) = \frac{e^{-\frac{i\pi}{8}}}{16\sqrt{\pi}}\int_{ \searrow ip}\frac{e^{-\frac{\pi i}{2}(v-ip)^2}}{\cosh\frac{\pi v}{2}} \left(-\frac{1}{3}\left[\sqrt{\pi}(v-ip)\right]^3-2i\sigma\sqrt{\pi}(v-ip)\right)dv.\end{equation*}
			Using the definition \eqref{20241016_13} for  $C_0(p)$,  we find that
			\begin{equation}\label{20241016_16} C_1(p) =\frac{1}{12\pi^2}C_0'''(p)+\frac{(1-2\sigma) }{4 i\pi}C_0'(p). 
		\end{equation}
		Since $C_0(p)$ is an even function, 
		$C_1(p)$ is an odd function. Hence, to find a bound of $|C_1(p)|$ for $p\in[-1,1]$, it is sufficient to consider $|C_1(p)|$ when $p\in [0,1]$. 
		Using the explicit expression for $C_0(p)$ given by \eqref{20241016_15}, we can compute $C_1(p)$ by \eqref{20241016_16}. Figure \ref{fig5} and Figure \ref{fig4} show  the graphs of $|C_1(p)|$, $p\in [0,1]$ when $\sigma=0$ and $\sigma=1$ respectively.
		
		\begin{figure}[ht]

			\begin{center}  \includegraphics[scale=0.5]{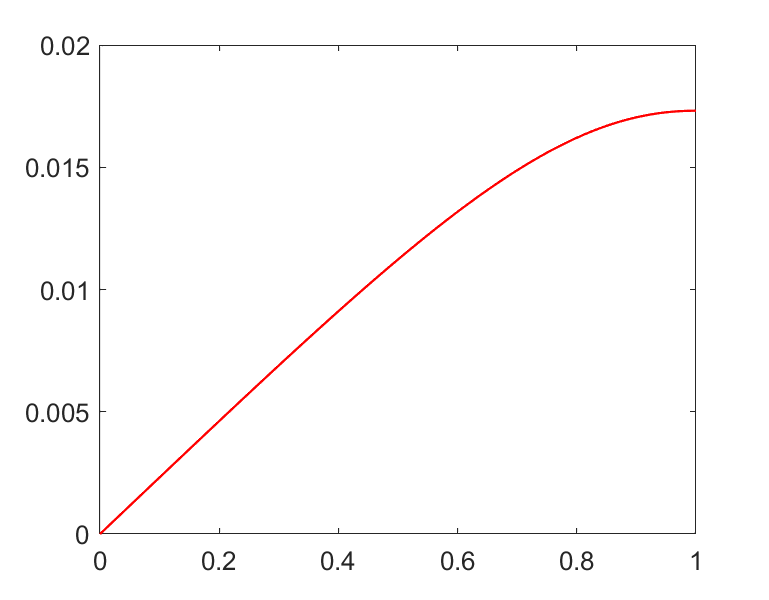}\end{center}
			
			\caption{The function $|C_1(p)|$ when $\sigma=0$ and  $0\leq p\leq 1$.\label{fig5}}
		\end{figure}
		
		\begin{figure}[ht]
			
			\begin{center}  \includegraphics[scale=0.5]{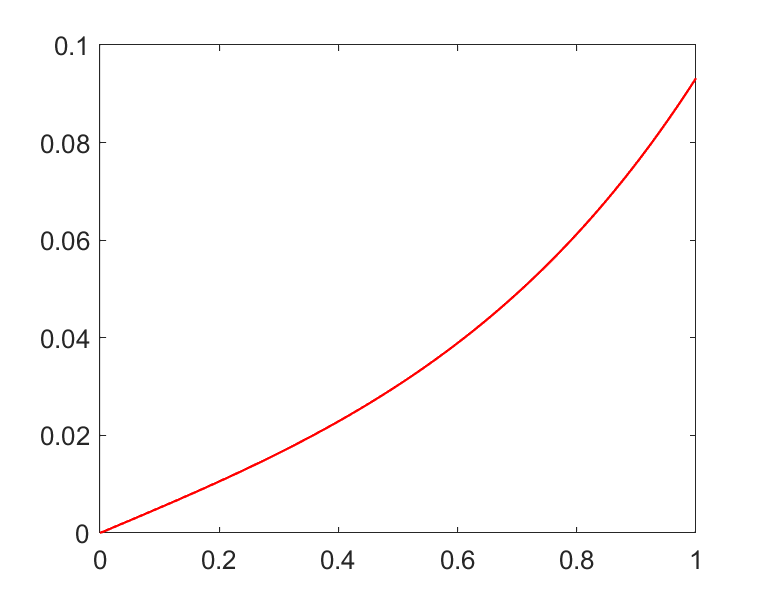}\end{center}
			
			\caption{The function $|C_1(p)|$ when $\sigma=1$ and  $0\leq p\leq 1$.\label{fig4}}

		\end{figure}
		
		From Figure \ref{fig5}, we find that when $\sigma=0$, $|C_1(p)|$ is increasing on $[0,1]$. Hence,  
		\[b_1(0)=|C_1(1)_{\sigma=0}|=0.0173.\]
		From Figure \ref{fig4}, we find that when $\sigma=1$, $|C_1(p)|$ is increasing on $[0,1]$. Hence, 
		\[b_1(1)=|C_1(1)_{\sigma=1}|=0.0932.\]
		This completes the proof. \qedhere
		
	\end{proof}
	In \cite{Patel2022}, Theorem 4.1 in \cite{AriasdeReyna_2011} is quoted directly to give the bounds
	\[b_1(0)\leq \frac{1}{\pi\sqrt{2(3-2\log 2)}}=0.1772,\hspace{1cm}b_1(1)\leq \frac{9}{2\sqrt{2\pi}}=1.7952.\]
	Obviously, our exact results \eqref{20250606_6} are much better.

	Finally, we consider 
	\[c(\sigma)=\max_{-1\leq p\leq 1}|tRS_1(p)|,\]where $RS_1(p)$ is defined by \eqref{20250606_1}.
	 In  the work \cite{AriasdeReyna_2011}, Arias de Reyna has obtained a bound $RS_K(p)$ for all $K\geq 1$. We specialize to the case where $K=1$ and obtain a better upper bound here.
	
	\begin{theorem}\label{20241103_5}
		Let
		\[c(\sigma) =\max_{-1\leq p\leq 1}|tRS_1(p)|.\]Then
		\[c(0)\leq 0.9704,\hspace{1cm}c(1)\leq 1.0450.\]
	\end{theorem}
	\begin{proof}
		As in \cite{AriasdeReyna_2011}, let 
		\[f(u)=-\frac{1}{2}-\frac{1}{ u}-\frac{1}{ u^2}\log(1-u).\]
		 Then by   (4.9) in \cite{AriasdeReyna_2011}, we have \begin{equation*}
			|tRS_1(p)| \leq  \frac{1}{\pi^2}\int_{-\infty}^{\infty} H(\sigma,y)dy,
		\end{equation*}
		where\[H(\sigma, y)=|1-u(y)|^{-\sigma}|u(y)|^{-2}\frac{1}{ 1+V(u(y))},\]
		\[u(y)=\frac{1}{2}+ye^{\frac{i\pi}{4}},\hspace{1cm} V(u)=\text{Re}\,f(u).\]
		Using numerical calculations, we obtain
		\[\frac{1}{\pi^2}\int_{-\infty}^{\infty}H(0,y) dy=0.9704,\hspace{1cm}\frac{1}{\pi^2}\int_{-\infty}^{\infty}H(1,y) dy=1.0450. \qedhere\]

	\end{proof}
	 In \cite{Patel2022}, Patel cited Theorem 4.2 in  \cite{AriasdeReyna_2011} with $K=1$ to obtain 
	\[c(0)\leq \frac{121\pi}{350}=1.0861,\hspace{1cm} c(1)\leq \frac{242\sqrt{2}\pi}{350}=3.0719.\]Our explicit computations in Theorem \ref{20241103_5} give smaller upper bounds.

	Now we return to the estimate of $|\zeta(1+it)|$.
	From \eqref{20241026_9}, Theorems \ref{231217_7}, \ref{20241103_3},  \ref{20241103_4}, \ref{20241103_5}, we find  that when $t>0$,
	\begin{equation}\label{20241027_3}
		|\zeta(1+it)| 
		\leq \left|\sum_{n\leq \sqrt{t/(2\pi)}}\frac{1}{n^{1+it}}\right| +\kappa_1(t)\left| \sum_{n\leq \sqrt{t/(2\pi)}}\frac{1}{n^{-it}}\right|+\kappa_2(t),\end{equation}
	where
	\[\kappa_1(t)= \sqrt{\frac{2\pi}{t}}\exp\left(\frac{\pi}{32t}-\frac{1}{24t^2}+\frac{5}{24t^4}\right)\frac{1}{1-e^{-\pi t}},\]
	\begin{align*}\kappa_2(t)&=\sqrt{\frac{2\pi}{t}}\left(\frac{1}{2}+b(1)\sqrt{\frac{2\pi}{t}}+\frac{   c(1)}{ t}\right)+\kappa_1(t)\left(\frac{1}{2}+b(0)\sqrt{\frac{2\pi}{t}}+\frac{  c(0)}{ t}\right),\end{align*}with
	\begin{align*}b(0)= 0.0173,\quad  b(1)= 0.0932,\quad 
		c(0)=0.9704,\quad c(1)= 1.0450.
	\end{align*}

	Notice that 
	\[\kappa_2(t)=O\left(\frac{1}{\sqrt{t}}\right).\]
	 For the two sums over $n\leq a=\sqrt{t/(2\pi)}$, using exponential sums with second order derivatives \eqref{20250529_3}, one would not be able to  obtain   estimates that are better than the crude estimates
	\[  \left|\sum_{n\leq a}\frac{1}{n^{1+it}}\right|\leq\sum_{n\leq a}\frac{1}{n}\leq \log a+\gamma+\frac{1}{a}\hspace{1cm} \text{and}\hspace{1cm}
	\left|\sum_{n\leq a}\frac{1}{n^{it}}\right|\leq a.\]
	Thus, we use these crude estimates and obtain
	\begin{align*}|\zeta(1+it)| 
		&\leq\frac{1}{2}\log t+\gamma-\frac{1}{2}\log(2\pi)+\vartheta(t),\end{align*}
	where
	\begin{align*}
		\vartheta(t)=\sqrt{\frac{2\pi}{t}}+\exp\left(\frac{\pi}{32t}-\frac{1}{24t^2}+\frac{5}{24t^4}\right)\frac{1}{1-e^{-\pi t}} +\kappa_2(t).
	\end{align*}It is easy to see that $\vartheta(t)$ is  decreasing when $t\geq 1$. Hence, for any $t_0\geq 1$,  if we let \begin{equation}\label{20250914_8}C=\gamma-\frac{1}{2}\log(2\pi)+\vartheta(t_0),\end{equation}
then
	\[|\zeta(1+it)|\leq\frac{1}{2}\log t+C\hspace{1cm}\text{when}\;t\geq t_0.\]
	In Table \ref{tab2}, we list down the values of $C$ for different $t_0$. 
	
	\begin{table}[ht]\caption{The values of $C$ \eqref{20250914_8} for different $t_0$.\label{tab2}}
	\begin{tabular}{|c|c|}
	\hline
		$t_0$	& 	$C$	\\\hline
	$10^{1}	$&$	2.4868	$\\\hline
	$10^{2}	$&$	1.1727	$\\\hline
	$10^{3}	$&$	0.8178	$\\\hline
	$10^{4}	$&$	0.7085	$\\\hline
	$10^{5}	$&$	0.6741	$\\\hline
	$10^{6}	$&$	0.6633	$\\\hline
	$10^{7}	$&$	0.6599	$\\\hline
	$10^{8}	$&$	0.6588	$\\\hline
	$10^{9}	$&$	0.6584	$\\\hline
        $	\hspace{0.5cm}10^{10}\hspace{0.5cm}	$&$	\hspace{0.5cm} 0.6583\hspace{0.5cm}	$\\\hline

	\end{tabular}
	\end{table}

	Since
	\[\lim_{t\to\infty}\vartheta(t)=1 \hspace{1cm}
	 \text{and}
	\hspace{1cm}\gamma-\frac{1}{2}\log(2\pi)= -0.3417,\] we find that
	\[C\geq \gamma-\frac{1}{2}\log(2\pi)+1=0.6583.\] In other words,
	we cannot use this method to yield a bound for $|\zeta(1+it)|$ that is better than
	\[\frac{1}{2}\log t+1-0.3417=\frac{1}{2}\log t+0.6583.\]
	Since we only perform numerical calculations of $\zeta(1+it)$ for $t\leq 10^6$, we take $t_0=10^6$. In this case, $C =0.6633$, 
	and we have
	\[|\zeta(1+it)|\leq \frac{1}{2}\log t+0.6633\hspace{1cm}\text{for}\; t\geq 10^6.\]
	The same code in Section \ref{sec2} is used for the numerical calculations of $|\zeta(1+it)|$ and for plotting Figure \ref{fig7}. The graphs show that for $e\leq t \leq 10^6$, $|\zeta(1+it)|$ is bounded above by $\frac{1}{2}\log t+0.6633$. Thus, 
\[|\zeta(1+it)|\leq \frac{1}{2}\log t+0.6633\hspace{1cm}\text{for}\; e\leq t\leq 10^6.\]
	\begin{figure}[ht]
		\begin{center}
			\includegraphics[scale=0.34]{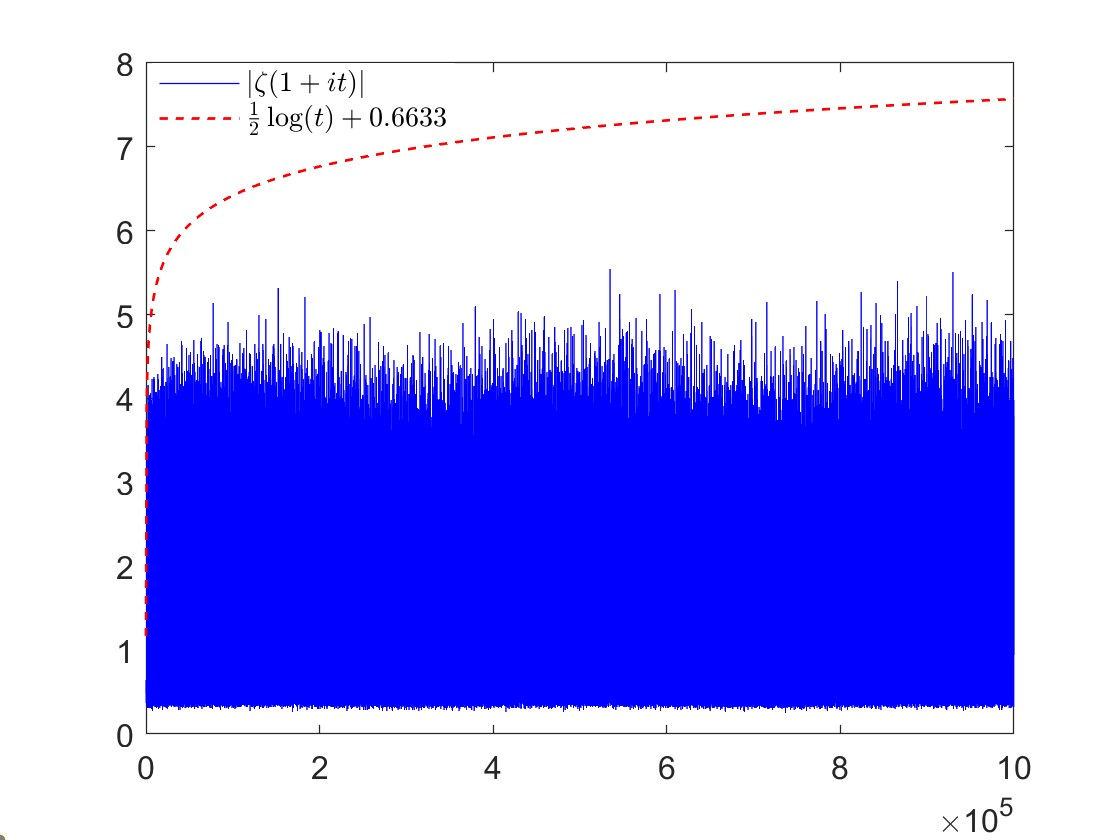}  \includegraphics[scale=0.34]{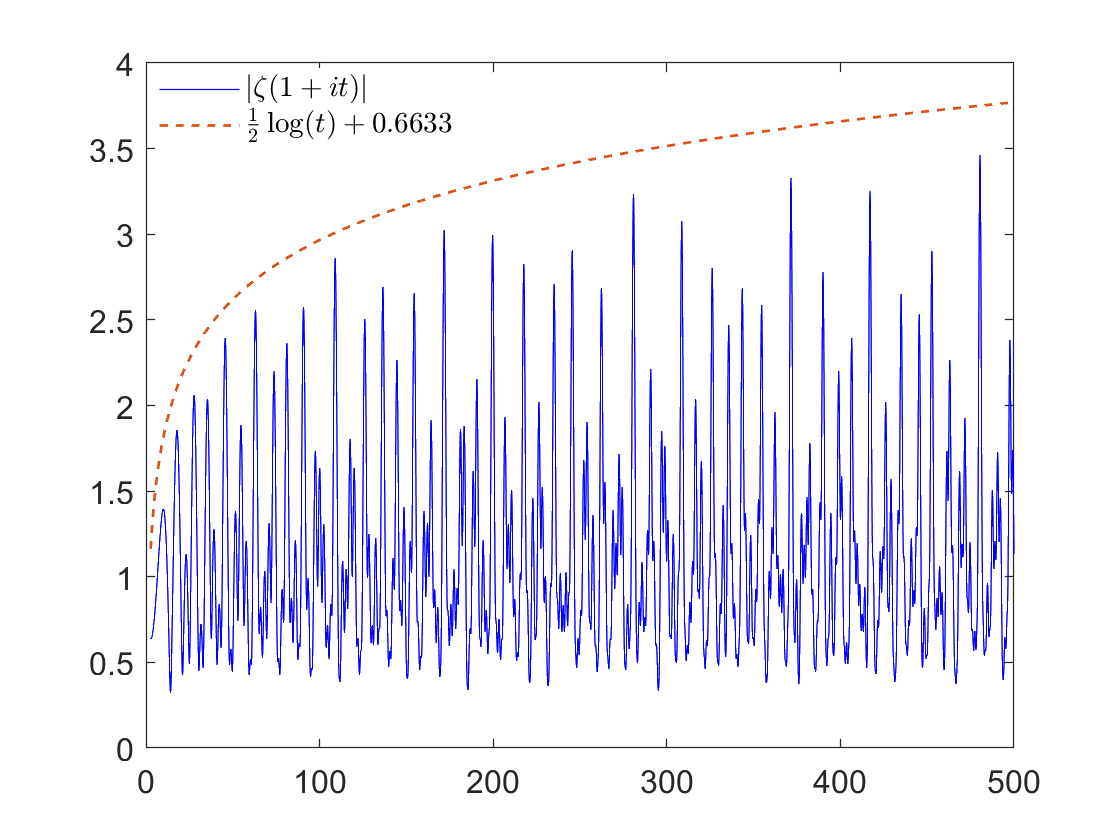}
		\end{center}
		\caption{The figures show comparisons of $|\zeta(1+it)|$ with $\frac{1}{2}\log t+0.6633$ when $e\leq t\leq 10^6$ and  when $e\leq t\leq 500$.\label{fig7}}
	\end{figure}
Hence, we obtain the following theorem.
	\begin{theorem}\label{20250613_1}For $t\geq e$,
	\begin{equation}\label{20241103_8}|\zeta(1+it)|\leq \frac{1}{2}\log t+0.6633.\end{equation}
	\end{theorem}
	
	\begin{remark}
	  In \cite{Patel2022},  Patel  directly cited the results of \cite{AriasdeReyna_2011} for bounds of   $b(0)$, $b(1)$, $c(0)$ and $c(1)$.  We have obtained exact  values for $b(0)$, $b(1)$, and obtained better bounds for $c(0)$ and $c(1)$ using direct computations. These improvements are not that significant for improving the bounds for $|\zeta(1+it)|$ when $t$ is large, as they only affect terms of order $t^{-\frac{1}{2}}$. The reason Patel only obtained the bound $\frac{1}{2}\log t+1.93$ is that he has only used $t_0=47.47$, as this is the point where the bound $\frac{1}{2}\log t+1.93$ is better than the Backlund's bound $\log t$. If we take $t_0=47.47$, we will get the bound $\frac{1}{2}\log t+1.4184$, which is  still better than Patel's result.
	\end{remark}
	 
		\bigskip
	\section{Concluding the Proof of Theorem \ref{maintheorem}}
	As in Section \ref{secexponential}, we are interested in a bound of the form $|\zeta(1+it)|\leq v\log t$ for all $t\geq e$. 
	As mentioned in Section \ref{sec2}, for $e\leq t\leq 10^6$, numerical calculations show that
	    $|\zeta(1+it)|/\log t$ achieves its maximum value 0.6443 when $t=17.7477$. 
	    To show that
	$|\zeta(1+it)|\leq 0.6443 \log t$ holds for all $t\geq e$, we can use Theorem \ref{20250613_1}.
	Notice that when $t\geq 100$,
	\[  \frac{1}{2}\log t+0.6633\leq  0.6443 \log t.\]
	
	\begin{figure}[ht]
		\begin{center}
			\includegraphics[scale=0.4]{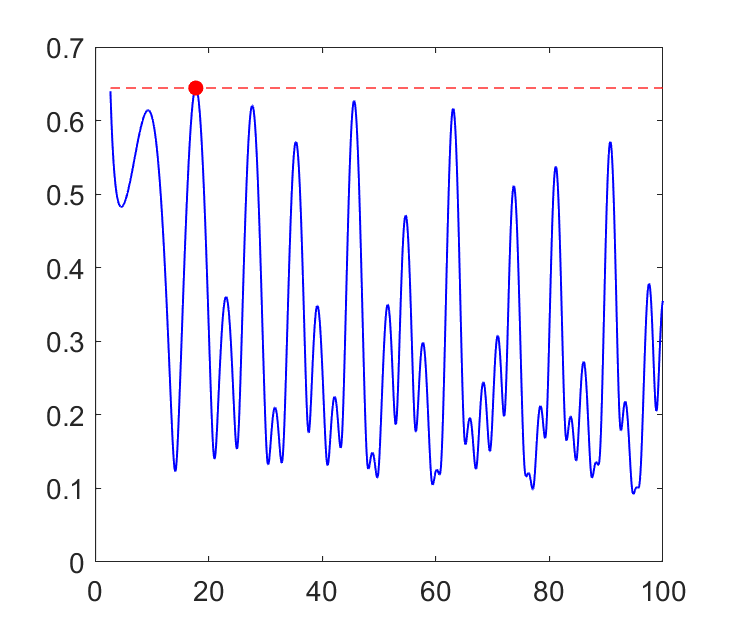}  \includegraphics[scale=0.4]{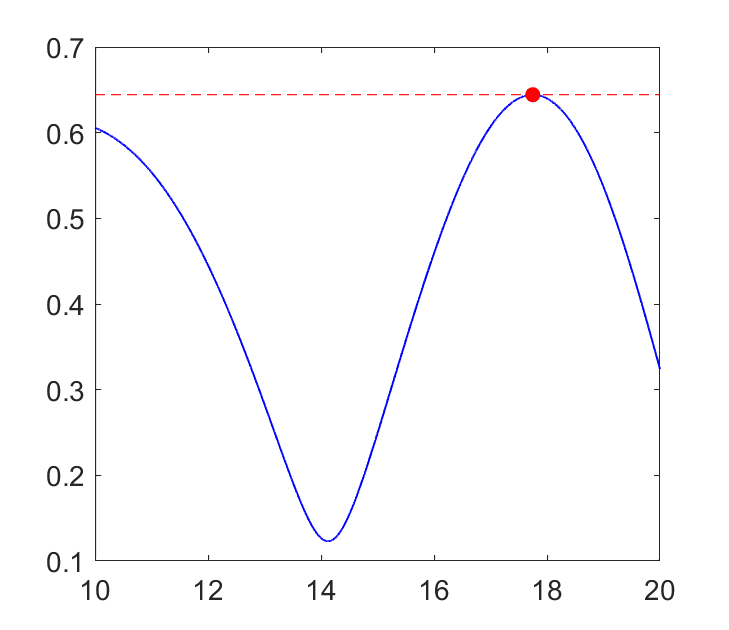}
		\end{center}
		\caption{The figures show $|\zeta(1+it)|/\log t$  when $e\leq t\leq 100$ and  when $10\leq t\leq 20$.\label{fig6}}
	\end{figure}
	
	In Figure \ref{fig6}, the graphs of  $|\zeta(1+it)|/\log t$  when $e\leq t\leq 100$ and  when $10\leq t\leq 20$ are shown. The values of $|\zeta(1+it)|$ are computed using the same code in Section \ref{sec2}. The graphs show that for $e\leq t\leq 100$, $|\zeta(1+it)|/\log t$ has a maximum value of $0.6443$. Thus,
we obtain the following theorem.
 \begin{theorem}When $t\geq e$, 
 \[|\zeta(1+it)|\leq 0.6443 \log t.\]
 The equality is achieved when $t=17.7477$. 
	    
 \end{theorem}

	
	\begin{table}[ht]\caption{The values of   $v$ \eqref{20250611_8} and $\tilde{v}$ \eqref{20250914_10} for given $t_0$.\label{tab3}}
\begin{tabular}{|c|c|c|}
\hline
	$t_0$	&	$v$	&	$\tilde{v}	$\\\hline
$	10^{5}	$&$	0.8134	$&$	0.5576	$\\\hline
$	10^{6}	$&$	0.7421	$&$	0.5480	$\\\hline
$	10^{7}	$&$	0.7003	$&$	0.5412	$\\\hline
$	10^{8}	$&$	0.6726	$&$	0.5360	$\\\hline
$	10^{9}	$&$	0.6526	$&$	0.5320	$\\\hline
$	10^{10}	$&$	0.6370	$&$	0.5288	$\\\hline
$	10^{11}	$&$	0.6245	$&$	0.5262	$\\\hline
$	10^{12}	$&$	0.6141	$&$	0.5240	$\\\hline
$	10^{13}	$&$	0.6053	$&$	0.5222	$\\\hline
$	10^{14}	$&$	0.5978	$&$	0.5206	$\\\hline
$	10^{15}	$&$	0.5912	$&$	0.5192	$\\\hline
$	10^{20}	$&$	0.5684	$&$	0.5144	$\\\hline
$	10^{30}	$&$	0.5456	$&$	0.5096	$\\\hline
$	10^{40}	$&$	0.5342	$&$	0.5072	$\\\hline
$	10^{50}	$&$	0.5274	$&$	0.5058	$\\\hline
$	10^{60}	$&$	0.5228	$&$	0.5048	$\\\hline
$	10^{70}	$&$	0.5196	$&$	0.5041	$\\\hline
$	10^{80}	$&$	0.5171	$&$	0.5036	$\\\hline
$	10^{90}	$&$	0.5152	$&$	0.5032	$\\\hline
$	10^{100}	$&$	0.5137	$&$	0.5029	$\\\hline
$	10^{200}	$&$	0.5068	$&$	0.5014	$\\\hline
$	\hspace{0.5cm}10^{300}	\hspace{0.5cm}$&$\hspace{0.5cm}	0.5046	\hspace{0.5cm}$&$	\hspace{0.5cm} 0.5010	\hspace{0.5cm}$\\\hline

\end{tabular}

\end{table}
	Finally, we   use Theorem \ref{20250613_1} to refine our results in Section \ref{secexponential} in the following way. For $t\geq t_0\geq e$,
	\[\frac{1}{2}+\frac{0.6633}{\log t}\leq \frac{1}{2}+\frac{0.6633}{\log t_0}.\]
	Hence, if \begin{equation}\label{20250914_10}\tilde{v}=\frac{1}{2}+\frac{0.6633}{\log t_0},\end{equation}then
	\[|\zeta(1+it)|\leq\frac{1}{2}\log t+0.6633\leq \tilde{v}\log t\hspace{1cm}\text{when}\;t\geq t_0.\]

In Table \ref{tab3}, we list down the values of $\tilde{v}$ and compare to the values of $v$ obtained in Section \ref{secexponential} for various values of  $t_0$. We see that the values of  $\tilde{v}$ are always smaller than the values of $v$. For $t_0=10^6$, we find that
\[|\zeta(1+it)|\leq 0.5480\log t\hspace{1cm}\text{when}\;t\geq 10^6.\]
Using the code in Section \ref{sec2} to numerically compute $|\zeta(1+it)|$ for $e\leq t\leq 10^6$, we find that when $  652.3704\leq t\leq 10^6$, we also have $|\zeta(1+it)|\leq 0.5480\log t$. This is sharp as the equality is achieved when $t=652.3704$. Therefore,  
\[|\zeta(1+it)|\leq 0.5480\log t\hspace{1cm}\text{when}\;t\geq  652.3704.\]
	 This result is better than $\frac{1}{2}\log t+0.6633$ when $652.3704\leq t\leq 10^6$.
This concludes the proof of our main results given in Theorem \ref{maintheorem}. 
	
	\bigskip
	\bibliographystyle{plain}
	\bibliography{ref}

\begin{thebibliography}{10}

\bibitem{Andrews}
G.~E. Andrews, R.~Askey, and R.~Roy.
\newblock {\em Special functions}, volume~71 of {\em Encyclopedia Math. Appl.}
\newblock Cambridge University Press, Cambridge, 1999.

\bibitem{Apostol_1}
T.~M. Apostol.
\newblock {\em Introduction to analytic number theory}.
\newblock Undergraduate Texts in Mathematics. Springer-Verlag, New
  York-Heidelberg, 1976.

\bibitem{AriasdeReyna_2011}
J.~Arias~de Reyna.
\newblock High precision computation of {R}iemann's zeta function by the
  {R}iemann-{S}iegel formula, {I}.
\newblock {\em Math. Comp.}, 80(274):995--1009, 2011.

\bibitem{Backlund}
R.~J. Backlund.
\newblock {S}ur les zéros de la fonction $\zeta(s)$ de {R}iemann.
\newblock {\em C. R. Math. Acad. Sci. Paris}, 158:1979--1982, 1914.

\bibitem{Chandrasekharan}
K.~Chandrasekharan.
\newblock {\em Introduction to analytic number theory}, volume Band 148 of {\em
  Die Grundlehren der mathematischen Wissenschaften}.
\newblock Springer-Verlag New York, Inc., New York, 1968.

\bibitem{ChengGraham2004}
Y.~F. Cheng and S.~W. Graham.
\newblock Explicit estimates for the {R}iemann zeta function.
\newblock {\em Rocky Mountain J. Math.}, 34(4):1261--1280, 2004.

\bibitem{Edwards}
H.~M. Edwards.
\newblock {\em Riemann's zeta function}.
\newblock Dover Publications, Inc., Mineola, NY, 2001.
\newblock Reprint of the 1974 original.

\bibitem{Ford_2002}
K.~Ford.
\newblock Vinogradov's integral and bounds for the {R}iemann zeta function.
\newblock {\em Proc. Lond. Math. Soc. (3)}, 85(3):565--633, 2002.

\bibitem{Francis}
F.~J. Francis.
\newblock An investigation into explicit versions of {B}urgess' bound.
\newblock {\em J. Number Theory}, 228:87--107, 2021.

\bibitem{Kolesnik_book}
S.~W. Graham and G.~Kolesnik.
\newblock {\em van der {C}orput's method of exponential sums}, volume 126 of
  {\em London Math. Soc. Lecture Note Ser.}
\newblock Cambridge University Press, Cambridge, 1991.

\bibitem{Yang_3}
G.~A. Hiary, N.~Leong, and A.~Yang.
\newblock Explicit bounds for the {R}iemann zeta function on the 1-line.
\newblock {\em Funct. Approx. Comment. Math.}, 2025.

\bibitem{Yang_2}
G.~A. Hiary, D.~Patel, and A.~Yang.
\newblock An improved explicit estimate for {$\zeta(1/2+it)$}.
\newblock {\em J. Number Theory}, 256:195--217, 2024.

\bibitem{Ivic_Book}
A.~Ivi\'{c}.
\newblock {\em The {R}iemann zeta-function, theory and applications}.
\newblock Dover Publications, Inc., Mineola, NY, 2003.

\bibitem{Kuzmin}
R.~O. Kuz'min.
\newblock Sur quelques in{\'e}galit{\'e}s trigonom{\'e}triques.
\newblock {\em Journal de la Soci{\'e}t{\'e} Physico-Math{\'e}matique de
  L{\'e}ningrade}, 1(2):233--239, 1927.

\bibitem{Landau_03}
E.~Landau.
\newblock Neuer {B}eweis des {P}rimzahlsatzes und {B}eweis des
  {P}rimidealsatzes.
\newblock {\em Math. Ann.}, 56(4):645--670, 1903.

\bibitem{Landau_28}
E.~Landau.
\newblock \"{U}ber einer trigonometrische {S}ummen.
\newblock {\em Nachr. Ges. Wiss. Gottingen}, pages 21--24, 1928.

\bibitem{Lehmer56}
D.~H. Lehmer.
\newblock Extended computation of the {R}iemann zeta-function.
\newblock {\em Mathematika}, 3:102--108, 1956.

\bibitem{Littlewood_1925}
J.~E. Littlewood.
\newblock On the {R}iemann {Z}eta-{F}unction.
\newblock {\em Proc. Lond. Math. Soc. (2)}, 24(3):175--201, 1925.

\bibitem{Mellin_1902}
H.~Mellin.
\newblock Eine {F}ormel {F}\"{u}r {D}en {L}ogarithmus {T}ranscendenter
  {F}unctionen von {E}ndlichem {G}eschlecht.
\newblock {\em Acta Math.}, 25(1):165--183, 1902.

\bibitem{Montgomery_Book}
H.~L. Montgomery and R.~C. Vaughan.
\newblock {\em Multiplicative number theory. {I}. {C}lassical theory},
  volume~97 of {\em Cambridge Studies in Advanced Mathematics}.
\newblock Cambridge University Press, Cambridge, 2007.

\bibitem{Patel2022}
D.~Patel.
\newblock An explicit upper bound for {$|\zeta (1+it) |$}.
\newblock {\em Indag. Math. (N.S.)}, 33(5):1012--1032, 2022.

\bibitem{Platt_Trudgian}
D.~J. Platt and T.~S. Trudgian.
\newblock An improved explicit bound on {$|\zeta(\frac 12+it)|$}.
\newblock {\em J. Number Theory}, 147:842--851, 2015.

\bibitem{Siegel1932}
C.~L. Siegel.
\newblock Uber {R}iemann's {N}achlass zur analytischen {Z}ahlentheorie,
  {Q}uellen und {S}tudien zur {G}eschichte der {M}athematik.
\newblock {\em Astronomie und Physik}, 2:45--80, see arXiv:1810.05198 for an
  English translation by E. Barkan and D. Sklar, 1932.

\bibitem{Titchmarsh}
E.~C. Titchmarsh.
\newblock {\em The theory of the {R}iemann zeta-function}.
\newblock The Clarendon Press, Oxford University Press, New York, second
  edition, 1986.
\newblock Edited and with a preface by D. R. Heath-Brown.

\bibitem{Trudgian_2014}
T.~Trudgian.
\newblock A new upper bound for {$|\zeta(1+it)|$}.
\newblock {\em Bull. Aust. Math. Soc.}, 89(2):259--264, 2014.

\bibitem{Vinogradov_58}
I.~M. Vinogradov.
\newblock A new estimate of the function {$\zeta (1+it)$}.
\newblock {\em Izv. Akad. Nauk SSSR Ser. Mat.}, 22:161--164, 1958.

\bibitem{Weyl_1921}
H.~Weyl.
\newblock {Z}ur {A}bsch\"atzung von $\zeta(1+ti)$.
\newblock {\em Math. Z.}, 10:88--101, 1921.

\bibitem{Yang_1}
A.~Yang.
\newblock Explicit bounds on {$\zeta(s)$} in the critical strip and a zero-free
  region.
\newblock {\em J. Math. Anal. Appl.}, 534(2):Paper No. 128124, 53, 2024.

\end{thebibliography}
\end{document}